\documentclass[12pt]{amsart}

\usepackage[numbers]{natbib}
\usepackage[utf8]{inputenc} % encoding
\usepackage[T1]{fontenc} % For Kestutis Cesnavicius
\usepackage[dvipsnames, table]{xcolor} % colors
\usepackage{color} % colors
\usepackage{tikz} % code figures
\usetikzlibrary{backgrounds, positioning}
\usepackage{tikz-cd} % commutative diagrams
\usetikzlibrary{decorations.pathmorphing}
\usepackage{amssymb, amsmath, mathtools, amsfonts, amsthm} % ams    
\usepackage{graphicx} % optional arguments for \includegraphics
\graphicspath{ {./images/} }
\usepackage{caption} % customize the captions in floating environments
\usepackage{latexsym} % for \lhd and \unlhd (and others)
\usepackage[inline]{enumitem} % customize items in lists
\usepackage{hyperref} % hyperlinks and references
\hypersetup{colorlinks=true, urlcolor=caribbeangreen, citecolor=gray, linkcolor=caribbeangreen}
\usepackage{url} 
\usepackage{indentfirst} % to indent the first paragraph of a section
\usepackage{listings} % source-code printing
\usepackage[mathscr]{eucal} % nice caligraphy for sheaves and categories
\usepackage{manfnt} % street sign
\usepackage{longtable} % tables that continue on the next page
\usepackage{colonequals} % correct way to write :=
\usepackage{nicefrac} % typeset in-line fractions in a "nice" way. \nicefrac[\mathcal]{numerator}{denominator}
\usepackage{caption} % figures and subfigures
\usepackage{subcaption} % figures and subfigures
\usepackage{float}
\usepackage[capitalize]{cleveref} % clever references
\usepackage{stackengine} % for contracted product command
\usepackage{breakcites}

\usepackage{stmaryrd} % Provides \fatslash
\usepackage{bm}       % Provides \pmb (bold math symbols)

\newcounter{counter}[section] % global counter / restarts with section.
\counterwithin{counter}{section}
\newtheorem{theorem}[counter]{Theorem}
\newtheorem{lemma}[counter]{Lemma}

\newtheorem{proposition}[counter]{Proposition}

\theoremstyle{definition} 
\newtheorem{definition}[counter]{Definition}
\newtheorem{situation}[counter]{Situation}

\newtheorem{example}[counter]{Example} % examples
\newtheorem{remark}[counter]{Remark}

\newtheorem{warning}[counter]{\textdbend Warning}

\definecolor{santicolor}{RGB}{0, 159, 183} % french blue
\definecolor{headercolor}{rgb}{0.83, 0.83, 0.83} % light gray
\definecolor{caribbeangreen}{rgb}{0.0, 0.8, 0.6}
\definecolor{acolor}{rgb}{0.01, 0.75, 0.24}
\definecolor{bcolor}{RGB}{0, 159, 183}
\definecolor{ccolor}{RGB}{254, 74, 73}
\definecolor{linkcolor}{rgb}{0.0, 0.0, 0.61}
\definecolor{codegreen}{rgb}{0,0.6,0}
\definecolor{codegray}{rgb}{0.5,0.5,0.5}
\definecolor{codepurple}{rgb}{0.58,0,0.82}
\definecolor{backcolour}{rgb}{0.95,0.95,0.92}
\lstdefinestyle{mystyle}{backgroundcolor=\color{backcolour}, commentstyle=\color
 {codegreen}, keywordstyle=\color{magenta}, numberstyle=\tiny\color{codegray}, stringstyle=\color{codepurple}, basicstyle=\ttfamily\footnotesize, breakatwhitespace=false, breaklines=true, captionpos=b, keepspaces=true, numbers=left, numbersep=4pt, showspaces=false, showstringspaces=false, showtabs=false, tabsize=2}
\newcommand{\Z}{\mathbb{Z}} % for Integers
\newcommand{\Q}{\mathbb{Q}} % for Rationals
\newcommand{\C}{\mathbb{C}} % for Complex numbers
\newcommand{\Qbar}{{\bar{\Q}}} % for algebraic numbers
\newcommand{\kbar}{{\bar{k}}} % algebraic closure of k
\newcommand{\OO}{\mathscr{O}} % for ring of integers
\newcommand{\mfp}{\mathfrak{p}} % ideal p
\newcommand{\rmp}{\mathrm{p}} % roman p

\newcommand{\A}{\mathbb{A}} % affine space
\newcommand{\Gm}{\mathbb{G}_\text{m}}
\newcommand{\kk}{{\mathbf k}} % function field - residue field
\renewcommand{\P}{\mathbb{P}} % projective space
\newcommand{\Pone}{\P^1} % projective line
\newcommand{\PP}{\mathscr{P}} % weighted projective stack
\newcommand{\XX}{\mathscr{X}} % mathscr X
\newcommand{\mcS}{\mathcal{S}} % mathcal S
\newcommand{\mcR}{\mathcal{R}} % mathcal R
\newcommand{\mcU}{\mathcal{U}} % mathcal U
\newcommand{\mcV}{\mathcal{V}} % mathcal V
\newcommand{\mcL}{\mathcal{L}} % mathcal L
\newcommand{\DD}{\mathbf{D}}
\newcommand{\belyi}{Belyi } % Belyi
\newcommand\HH{\mathrm{H}}
\newcommand\cHH{\check{\mathrm{H}}}
\newcommand{\fppf}{\texttt{fppf}} % faithfully flat and locally of finite presentation
\newcommand{\et}{\texttt{\'et}} % étale: flat + G-unramified
\newcommand{\ab}{\texttt{ab}} % abelianization
\newcommand{\tribar}{\bar\triangle} % triangle group
\newcommand{\x}{\mathsf{x}}
\newcommand{\y}{\mathsf{y}}
\newcommand{\z}{\mathsf{z}}
\renewcommand{\t}{\mathsf{t}}
\newcommand{\s}{\mathsf{s}}

\newcommand{\sfH}{\mathsf{H}} % mathsf H
\newcommand{\sfS}{\mathsf{S}} % mathsf S
\newcommand{\sfK}{\mathsf{K}} % mathsf K
\newcommand{\abc}{(a,b,c)}

\newcommand{\md}{\operatorname{mod}} % mod without parenthesis

\newcommand{\ideal}[1]{\langle #1 \rangle} % ideal
\newcommand{\cdef}[1]{{\color{linkcolor}\textsf{#1}}} % highlight words in definitions
\newcommand{\brk}[1]{\left\{ #1 \right\}} % brackets
\newcommand{\inv}{^{-1}}
\newcommand{\conprod}[1]{\stackon{$\times$}{\tiny $#1$}} % contracted product
\renewcommand{\th}{\text{th}}
\newcommand{\Sch}{\texttt{Sch}} % category of schemes
\newcommand{\GrpSch}{\texttt{GrpSch}} % category of schemes
\newcommand{\bfj}{\mathbf{j}}
\newcommand{\bfw}{\mathbf{w}}

\newcommand{\gfe}{A\x^{a} + B\y^{b} + C\z^{c}}

\newcommand{\Autsch}{\mathbf{Aut}}
\newcommand{\spcite}[1]{\cite[\href{https://stacks.math.columbia.edu/tag/#1}{Tag
    #1}]{stacks-project}}

\newcommand{\eclabel}[1]{\href{https://www.lmfdb.org/EllipticCurve/Q/#1}{\texttt{#1}}}

\DeclareMathOperator{\Spec}{Spec} % Spec
\DeclareMathOperator{\Proj}{Proj} % Proj
\DeclareMathOperator{\Gal}{Gal} % Galois
\DeclareMathOperator{\Stab}{Stab} % Stabilizer
\DeclareMathOperator{\Div}{Div}

\DeclareMathOperator{\Aut}{Aut} % Automorphism group
\DeclareMathOperator{\Hom}{Hom} % Automorphism group
\DeclareMathOperator{\Pic}{Pic}

\DeclareMathOperator{\id}{id}
\DeclareMathOperator{\lcm}{lcm}

\usepackage[textsize=tiny]{todonotes}
\usepackage{xargs, xpatch}
\usepackage{marginnote}

\title{Fermat descent}
\author{Santiago Arango-Pi{\~n}eros}

\begin{document}

\begin{abstract}
  Descent theory (a modern formulation of Fermat's classical method of \emph{infinite
    descent}) is a powerful tool in arithmetic geometry. In this
  article, we reinterpret descent theory through the lens of quotient stacks
  and apply it in the setting where it first arose: the Diophantine study of
  generalized Fermat equations
  \begin{equation}
      \label{eq:abstract-gfe}
      \gfe = 0.
  \end{equation}
  We focus on understanding the arithmetic of the stacks that arise from the
  study of primitive integral solutions to \Cref{eq:abstract-gfe}, rather than
  on solving any particular instance of the equation.
\end{abstract}

\maketitle
\vspace{-8mm}
\setcounter{tocdepth}{1}
\tableofcontents
\vspace{-11mm}

% \newpage
% \makeatletter
% \providecommand\@dotsep{5}
% \makeatother
% \listoftodos\relax

% \newpage
\section{Introduction}
\label{sec:introduction}

\subsection{The framework of Poonen--Schaefer--Stoll}
\label{sec:PSS}

Poonen, Schaefer, and Stoll \cite[Theorem 1.1]{PoonenSchaeferStoll07} provably
computed the finite set of primitive integral solutions of the generalized
Fermat equation
\begin{equation}
  \label{eq:pss}
  F\colon \x^2 + \y^3 + \z^7 = 0.
\end{equation}
(Recall that $(x,y,z) \in \Z^3$ is called \cdef{primitive} when
$\gcd(x,y,z) = 1$.) If $\mcU$ is the \cdef{punctured cone} associated to
$F \subset \A^3_\Z$ (i.e., the subscheme obtained by deleting $\brk{\x = \y =
  \z = 0}$ from $F$), then $\mcU(\Z)$ is identified with the set of primitive
integral solutions to \Cref{eq:pss}. A preliminary step in their method is to
consider the quotient stack $[\mcU/\Gm]$, where the multiplicative group $\Gm$
acts by
\[(x,y,z)\cdot \lambda \colonequals (\lambda^{21}x,\lambda^{14}y,
  \lambda^{6}z).\] After inverting the bad primes $\mcS = \brk{2,3,7}$, the
stack $[\mcU/\Gm]$ becomes isomorphic to the stack $\Pone(2,3,7)$; this is the
projective line $\Pone_\Z$ rooted at the irreducible horizontal divisors
$0, 1$, and $\infty$ with multiplicities $2,3,$ and $7$, respectively.
\begin{equation}
  \label{eq:pss-iso}
  [\mcU/\Gm]_{\Z[1/42]} \cong \Pone(2,3,7)_{\Z[1/42]}.
\end{equation}
After this preliminary result, the first step in the method is to find a
geometrically Galois Belyi map $\phi\colon X \to \Pone$, with ramification
indices $2,3,7$ above $0,1,\infty$. To find $\mcU(\Z)$, it is enough to
calculate the sets of rational points on the curves $X_\tau$ for an explicit
finite set of twists $\phi_\tau \colon X_\tau \to \Pone$ of the map
$\phi\colon X \to \Pone$. Our interpretation of the method of \cite[Section
3]{PoonenSchaeferStoll07} tersely summarized here is henceforth referred to as
\cdef{Fermat descent}.

\subsection{Our main theorem}
\label{sec:intro-main-theorem}
To generalize the method of Fermat descent to arbitrary generalized Fermat
equations
\begin{equation}
  \label{eq:gfe}
  F\colon A\x^a + B\y^b + C\z^c = 0, 
\end{equation}
one must start by finding the correct analog of the isomorphism~(\ref{eq:pss-iso}). This is
our main contribution.

Let $\mcS$ be the set of primes dividing the integer
$a\cdot b\cdot c\cdot A\cdot B\cdot C \neq 0$, and denote by $R$ the ring of
$\mcS$-integers. Let $\mcU$ be the punctured cone associated to \Cref{eq:gfe}.
% We show in \Cref{lem:H-structure} that the stabilizer subgroup of $\mcU_R$ in the
% torus
% \[\Gm^3 \colonequals \Spec R [t_0^{\pm 1}, t_1^{\pm 1}, t_\infty^{\pm 1}],\] is the affine
% subgroup scheme $\sfH_R$ defined by the ideal
% \[\ideal{t_0^a-t_1^b,t_1^b-t_\infty^c, t_\infty^c-t_0^a} \subset R [t_0^{\pm 1}, t_1^{\pm 1},
%   t_\infty^{\pm 1}].\]
Let $\sfH$ be the subgroup of $\Gm^3$ given on points by those
$(\lambda_0, \lambda_1, \lambda_\infty)$ such that
$\lambda_0^a = \lambda_1^b = \lambda_\infty^c$. The group $\sfH$ visibly acts on
$\mcU$ by coordinate-wise multiplication
\begin{equation*}
  (x,y,z)\cdot (\lambda_0, \lambda_1, \lambda_\infty) \colonequals (\lambda_0
  x, \lambda_1 y, \lambda_\infty z).
\end{equation*}
Finally, let $\Pone\abc$ denote the iterated root stack of $\Pone_\Z$ at the
divisors $0,1,\infty$ with multiplicities $a,b,c$. This is the stacky version
of Darmon's $M$-curve $\mathbf{P}^1_{a,b,c}$ \cite[p.~4]{Darmon97} (see \Cref{fig:belyi-stack}).
\begin{figure}[ht]
  \centering
  \includegraphics[width=\textwidth]{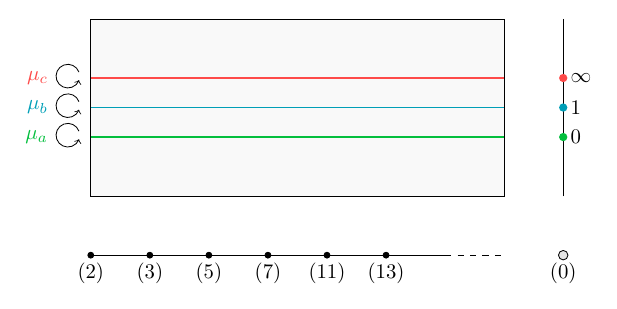}
  \caption{The Belyi stack of signature $\abc$.}
  \label{fig:belyi-stack}
\end{figure}
\begin{theorem}
  \label{thm:main} The map
  \begin{equation}
    \label{eq:j}
    j\colon \mcU_R \to \Pone_R, \quad (x,y,z) \mapsto (-Ax^a:Cz^c)
  \end{equation}
  induces an isomorphism $[\mcU_R/\sfH_R] \cong \Pone\abc_{R}$.
\end{theorem}

\begin{warning}
  For a general triple $\abc$ of positive integers, the group $\sfH$ is not
  isomorphic to $\Gm$. We show in \Cref{lem:H-structure} that this is only the
  case when $\gcd(bc,ab,ac) = 1$. For example, in the important case of
  $\abc = (p,p,p)$ for some prime $p$, the group scheme $\sfH$ is isomorphic to
  $\Gm\times \mu_p\times \mu_p$.
\end{warning}

To further motivate the method of Fermat descent and introduce notation, we
revisit one of the very first instances of the method of \emph{infinite
  descent} from this point of view.
\subsection{Fermat's last theorem for \texorpdfstring{$n=4$}{n=4}}
\label{sec:x4+y4-z2}
The term \emph{infinite descent} was coined in a
\href{https://rcin.org.pl/impan/dlibra/publication/198805/edition/180222/content}{letter}
  from Fermat to Carcavi:

\begin{quotation}
  \noindent\rule{10.2cm}{0.4pt}
  
  \noindent And because the ordinary methods found in the books were insufficient to
  prove such difficult propositions, I finally discovered a completely novel
  path to reach them. I called this method of proof \emph{infinite or indefinite
  descent}, etc.; at first, I used it only to prove negative propositions, such
  as, for example:

  ...
  %That there is no number, less than unity times a multiple of 3, which is made
  %up of a square and three times another square;

  \noindent That there is no right triangle with integer sides whose area is a
  square number.
  
  \noindent\rule{10.2cm}{0.4pt}
  
  \hfill Fermat, 1659
\end{quotation}
Fermat actually proved the claim above (see \cite[Chapter XXII, p.
615]{Dickson66}). Closely related to this problem is his famous ``Last
Theorem'' in the case of exponent $n=4$. In his controversial marginal notes to
the \emph{Arithmetica} of Diophantus, Fermat states that ``the sum of two
biquadrates is never a biquadrate or a square.''
% Even though the first complete proof of this statement---using the method of
% infinite descent---appears to be due to Euler \cite[Chapter~XXII,
% p.~618]{Dickson66}, the result is nevertheless attributed to Fermat.
(Excellent expositions of the proofs of these results via infinite descent are
given in \cite{Conrad-congruent, Conrad-descent}.)

At first glance, the method of infinite descent appears to be a simple reversed
form of the principle of mathematical induction. In fact, it is much deeper. We
present an overly complicated proof of Fermat's last theorem for exponent $n=4$
using the method of \cdef{Fermat descent} with the objective of illustrating the
hidden geometry and introducing notation. The foreign definitions and
constructions presented here will be introduced later on.

Fermat's last theorem for exponent $n=4$ follows from the following stronger
statement.
\begin{theorem}[Fermat]
  \label{thm:fermat-442}
 The primitive integral solutions $(x,y,z)$ to the generalized Fermat equation $F\colon
 \x^4 + \y^4 - \z^2 = 0$ satisfy $x\cdot y \cdot z = 0$. They are the eight triples
 \[ \pm(1,0,1), \pm (1,0,-1), \pm (0,1,1), \pm (0,1,-1).\]
\end{theorem}
\begin{proof}[Proof sketch]
  Let $\mcU$ be the punctured cone associated to $F$. Our goal is to show that
  $\mcU(\Z)$ contains only the eight elements listed above. The group scheme
  $\sfH$ defined above the statement of \Cref{thm:main} (and discussed in
  detail in \Cref{sec:H}) acts on $\mcU$ by coordinate-wise multiplication, and
  it is not hard to see there are non-trivial stabilizers. For instance, the
  $\Qbar$-points can be described explicitly as
  \begin{equation*}
    \sfH(\Qbar) = \brk{(\xi_0\lambda, \xi_1\lambda, \xi_\infty\lambda^2 ) :
      \lambda\in \Qbar^\times, \xi_0, \xi_1 \in \mu_4(\Qbar), \xi_\infty \in \mu_2(\Qbar)}.
  \end{equation*}
  From this description, we see that any point $(x,y,z)\in \mcU(\Qbar)$ with
  $x = 0$, $y = 0$, or $z = 0$ has $\mu_4(\Qbar), \mu_4(\Qbar),$ or
  $\mu_2(\Qbar)$ stabilizers in $\sfH(\Qbar)$, respectively. On the other hand,
  any point $(x,y,z) \in \mcU(\Qbar)$ with $x\cdot y\cdot z \neq 0$ has no
  non-trivial stabilizers.

  The idea is to understand the arithmetic of $\mcU$ by studying instead the
  arithmetic of the stack quotient $[\mcU/\sfH]$.
 
 \begin{remark}[Notation]
    If $\XX$ is a stack and $R$ is a ring, we denote by $\XX(R)$ the \emph{groupoid}
    of $R$-points, and by $\XX\ideal{R}$ the \emph{set} of $R$-points, (see
    \Cref{sec:conventions-stacks}).
  \end{remark}

  %Since $\sfH$  over $\Spec \Z$
  %is fppf (like any diagonal group scheme of multiplicative type),
  Descent theory (see \Cref{thm:descent}) gives the partition
  \begin{equation}
    \label{eq:3}
    [\mcU/\sfH]\ideal{\Z} = \bigsqcup_{\delta \in
      \HH^1(\Z,\sfH)}\mcU_\delta(\Z)/\sfH(\Z) \supset \mcU(\Z)/\sfH(\Z),
  \end{equation}
  where the (fppf) cohomology group $\HH^1(\Z,\sfH)$ is finite, and the
  $\mcU_\delta$ are certain twists of $\mcU$ arising from generalized Fermat
  equations $F_\delta$. 
  \begin{remark}
    The group $\sfH$ is not the only (fppf) group scheme acting on $\mcU$. For
    instance, one can also consider the image $\Gm(1,1,2)$ of the map
    $\Gm \to \Gm^3 \colon \lambda \mapsto (\lambda, \lambda, \lambda^2)$. The
    quotient $[\mcU/\Gm(1,1,2)]$ has the technical advantage of being a closed
    substack of the weighted projective stack $\PP(1,1,2)$. The reason for
    choosing $\sfH$ over $\Gm(1,1,2)$ will become clear shortly.
  \end{remark}

  \medskip
  \noindent \textbf{Step 0:} In this case, $p=2$ is the only bad prime, so
  $\mcS = \brk{2}$ and we abbreviate $R = \Z[1/2]$. From \Cref{thm:main}, we
  have an isomorphism of stacks
  \begin{equation}
    \label{eq:step-0-iso-441}
    [\mcU_R/\sfH_R] \cong \Pone(4,4,2)_R.
  \end{equation}
  
  The notable consequence of this isomorphism is that, by the definition of the
  \belyi stack $\Pone(4,4,2)$, geometrically Galois Belyi maps
  $\phi\colon X \to \Pone_\Q$ with signature $(4,4,2)$ factor through {\'e}tale
  covers $\psi\colon X \to \Pone(4,4,2)_\Q$, and thus we learn new information
  about $\mcU(\Z)$ by studying the sets of rational points $X(\Q)$ arising from
  such maps.

  \medskip
  \noindent \textbf{Step 1: (Covering)} The first task of the method of Fermat
  descent is to find a geometrically Galois Belyi map $\phi$ of signature
  $(4,4,2)$ with good reduction outside of $\mcS = \brk{2}$ (see
  \Cref{def:galois-belyi-map}). The
  \href{https://beta.lmfdb.org/Belyi/4T1/4/4/2.2/a/}{LMFDB (beta)}
  \cite{lmfdb:442-belyi-map} gave us the elliptic curve
  $E_\Q\colon v^2w = u^3-uw^2 \subset \P^2_\Q$ defined over $\Q$ with
  $j$-invariant $1728$ and the map
  \begin{equation}
    \label{eq:belyi-442}
    \phi\colon E_\Q \to \Pone_\Q, \quad (u:v:w) \mapsto (u^2:u^2-w^2).
  \end{equation}
  The same equations define an $R$-model $\Phi\colon E \to \Pone_R$. Let $\Autsch(\Phi)$
  be the automorphism $R$-group scheme of $\Phi$. In this
  situation, we have an isomorphism
  \begin{equation}
    \label{eq:2}
    \Pone(4,4,2)_R \cong [E/\Autsch(\Phi)].
  \end{equation}

  \medskip
  \noindent \textbf{Step 2: (Twisting)} It turns out that $\Autsch(\Phi) = \Autsch(E)$. In
  addition, $\Autsch(E) \cong \mu_4 = \Spec R[t]/\ideal{t^4-1}$ by \cite[Corollary
  III.10.2]{Silverman09}. By Kummer theory, we know that the cohomology group
  $\HH^1(R,\mu_4)$ is isomorphic to the finite group
  $R^\times/(R^\times)^4 \cong \brk{\pm 1, \pm 2, \pm 4, \pm 8}$. Once again,
  descent theory gives a partition
  \begin{align*}
    [E/\Autsch(\Phi)]\ideal{R} = \bigsqcup_{\tau \in \HH^1(R,\mu_4)}
                     \psi_\tau(E_\tau(R)) = \bigsqcup_{d \in R^\times/(R^\times)^4}
                     \psi_d(E_d(\Q)).
  \end{align*}
  The last equality follows because the twists $E_d$ are all proper and thus
  $E_d(R) = E_d(\Q)$ by the valuative criterion.

  The final remaining task in this step is to calculate the quartic twists
  $\Phi_d\colon E_d \to \Pone_R$ (see \Cref{sec:descent-theory-revisited}). By
  recalling that $\mu_4$ acts on the elliptic surface $E\colon v^2 = u^3- u$ via
  \begin{equation*}
    (u,v) \cdot \zeta \colonequals (\zeta^2u, \zeta^3 v),
  \end{equation*}
  and that for each $d \in \HH^1(R,\mu_4)$ the corresponding (left fppf)
  $\mu_4$-torsor is $T_d \colonequals \Spec R[t]/\ideal{t^4-d} \to \Spec R$
  with action
    \begin{equation*}
    \zeta\cdot \sqrt[4]{d} \colonequals \zeta\sqrt[4]{d},
  \end{equation*}
  an invariant calculation gives that $E_d\colon v^2w = u^3-duw^2 \subset \P^2_R$ and
  \begin{equation}
    \label{eq:belyi-442-d}
    \Phi_d \colon E_d\to \Pone_R, \quad (u:v:w) \mapsto (u^2:u^2-dw^2).
  \end{equation}
  This coincides with the Galois cohomology perspective in \citep[Proposition
  X.5.4]{Silverman09}. Indeed, it turns out that $\HH^1(R, \Autsch(\Phi))$ is isomorphic
  to the Galois cohomology group $\HH^1_{\brk{2}}(\Q, \mu_4(\Qbar))$
  parametrizing isomorphism classes of $\mu_4$-torsors over $\Q$ unramified
  outside $\{2\}$.

\medskip
\noindent \textbf{Step 3: (Sieving)} To summarize, we have shown that the set $j(\mcU(\Z))$ is
contained in
\begin{equation*}
  j(\mcU(R)) \subset \bigcup_{d\in R^\times/(R^\times)^4} \phi_d(E_d(\Q)).
\end{equation*}
Our objective now is to sieve out the points in $\mcU(R)$ that are not in
$\mcU(\Z)$. The best case scenario here would be that the sets $E_d(\Q)$ are
finite and that $\phi_d(E_d(\Q))$ are contained in
$\brk{0,1,\infty} \subset \Pone(\Q)$. Indeed, this would imply that
$j(\mcU(\Z)) \subset \brk{0,1,\infty}$, which can easily be seen to imply
\Cref{thm:fermat-442}. Unfortunately, this is not the case: the elliptic curves
$E_{2}$ and $E_{-8}$ (with LMFDB labels \eclabel{256.b1} and \eclabel{256.b2})
\nocite{lmfdb}have infinitely many $\Q$-rational points. We are forced to take
a closer look at the rational points on the projective line arising from
$\mcU(\Z)$ and $E_d(\Q)$ simultaneously.

For any choice of $d \in \brk{\pm 1, \pm 2, \pm 4, \pm 8}$, let us consider a
point $Q$ in the intersection
$j(\mcU(\Z)) \cap \phi_d(E_d(\Q)) \subset \Pone(\Q)$, so
\begin{equation*}
  Q = (x^4:z^2) = (u^2: u^2-dw^2)
\end{equation*}
for some primitive integral solution $(x,y,z)$ to $\x^4 + \y^4 = \z^2$, and
some rational point $P = (u:v:w) \in E_d(\Q)$. Let $(u,v,w) \in \Z^3$ be a
representative of $P$ with $\gcd(u,v,w) = 1$. The points $Q = 0, 1$ are the
expected ones:
\begin{itemize}[leftmargin=*]
\item $Q = 0$ forces $x = 0$, giving
  $(x,y,z) = \pm(0,1,1), \pm(0,1,-1)$, and
 \item  $Q = 1$ forces $y = 0$, giving
$(x,y,z) = \pm(1,0,1), \pm(1,0,-1)$.
\end{itemize}
(The value $Q = \infty$ cannot occur, since it
would force $z = 0$ and hence $x = y = 0$.) We may assume that $Q \neq 0, 1$;
equivalently, that $P \notin \brk{(0:0:1), (0:1:0)}$. In particular $u \neq 0$ and
$w \neq 0$. It follows that there is some $\lambda \in \Q^{\times}$ such that
$x^4 = \lambda u^2,$ and $z^2=\lambda(u^2-dw^2)$. The first identity implies that
$\lambda = \delta^2 \in (\Q^{\times})^2$. Moreover, $y^4 = z^2 - x^4 = -d\delta^2 w^2$. But this forces
$-d$ to be a square, so $d \in \brk{-1, -4}$. Fortunately, both
$E_{-1}(\Q)$ and $E_{-4}(\Q)$ are finite (see \eclabel{64.a4},
\eclabel{32.a4}), and a calculation gives:
\begin{align*}
  \phi_{-1}(E_{-1}(\Q)) = \brk{0, 1},\qquad  \phi_{-4}(E_{-4}(\Q)) = \brk{0, 1, (1:2)}.
\end{align*}
The points $Q = 0, 1$ are the expected ones. To rule out $Q = (1:2)$, note that
such a point satisfies $(x^4 : z^2) = (1:2)$, hence $z^2 = 2x^4$. Since
$(x,y,z)$ is a primitive solution, $x \neq 0$, and comparing $2$-adic
valuations gives a contradiction.
\end{proof}

\subsection{Summary of contributions}
\label{sec:summary}
This article contributes a thorough exposition of the stack-theoretic approach
to the Diophantine study of generalized Fermat equations, complemented by a
number of new results. The stack-theoretic perspective offers the distinct
advantage of being more conceptual, and we expect it will prove useful to
tackle numerous open problems in this area (see
\cite{SAP&CountingPrimitiveSolutions}.)

\begin{itemize}[leftmargin=*]
\item In \Cref{sec:quotient-stacks-descent}, we provide a concise review of
  quotient stacks and use it to present the fundamentals of \emph{descent
    theory}, in the sense of \cite[Chapter~2]{Skorobogatov01}, from this
  vantage point. The main result of this section is \Cref{thm:descent}, the
  \emph{descent theory partition}. It is stated in greater generality in
  \citep[Lemma 2.4]{Santens2023}. While this is certainly well known, we find our
  proof to be succinct and illuminating.
\item In \Cref{sec:root-stacks-belyi}, we begin by reviewing the root stack
  construction, which is due to Cadman \cite{Cadman07}. Our exposition borrows
  from Olsson’s treatment in \cite[Section~10.3]{Olsson16}. The purpose of this
  review is to understand the arithmetic of the \belyi stack $\Pone\abc$, which
  provides the stack-theoretic interpretation of Darmon’s $M$-curve
  $\mathbf{P}^1_{a,b,c}$ \cite[p.~4]{Darmon97}. It is well known to experts
  that Darmon’s $M$-curves can be interpreted as root stacks over their coarse
  moduli spaces. This perspective is hinted at by Poonen in
  \cite{Poonen_slides2006}, and addressed more directly by Santens in
  \href{https://mathoverflow.net/questions/390541/relation-between-stacky-curves-and-m-curves}{this
    MathOverflow post} \cite{SantensMO} and \cite[Lemma 2.1]{Santens2023}. In
  turn, this is an instance of a general feature of Deligne--Mumford stacks
  \cite{Geraschenko&Satriano17}. The main result of this section is
  \Cref{lemma:R-points-belyi-stack}, which explicitly characterizes the PID
  points on a \belyi stack.
\item In \Cref{sec:fundamental-theorem}, we prove our main result:
  \Cref{thm:re-main}. In spirit, this theorem captures a striking connection
  between:
  \begin{itemize}
  \item \textbf{Arithmetic:} the Diophantine
    equations $\gfe = 0$.
    \item \textbf{Geometry:} the Riemann sphere $\C\Pone$
      with orbifold points at $(0,1,\infty)$ of multiplicities $\abc$.
 \end{itemize}
  This connection was already identified by
  \cite{Darmon&Granville95}, and formulated in terms of stacks by
  \cite{Poonen_slides2006}, \cite{PoonenSchaeferStoll07},
  \cite[Example~5.4.7]{Voight&Zureick-Brown22}, and \cite{Poonen_slides2023}.
  The originality of our contribution lies in \Cref{sec:H}, where we work out
  the combinatorics of the general case, when $\gcd(bc, ac, ab)$ is possibly
  greater than one.
\end{itemize}

\subsection*{Acknowledgements}
\label{sec:acknowledgements}
This work is part of the author’s PhD thesis. We thank David Zureick-Brown,
John Voight, and Andrew Kobin for many enlightening conversations on this topic
and for their valuable feedback. We are grateful to Bjorn Poonen for serving on
the thesis committee and for his detailed and insightful comments on an earlier
draft. We thank Max Schwegele for a careful reading of an earlier version and
for numerous valuable corrections.

\section{Quotient stacks and descent}
\label{sec:quotient-stacks-descent}

\subsection{Conventions on stacks}
\label{sec:conventions-stacks}
Recall that a morphism of schemes is \cdef{fppf} if it is faithfully flat and
locally of finite presentation (see \cite[Definition 3.4.1]{Poonen2017}). For a
choice of base scheme $S$, we work on the \cdef{big fppf site}
$S_\fppf = (\Sch_{/S})_\fppf$. This is the category $\Sch_{/S}$ of schemes over
$S$ where the open coverings are families $\brk{U_i \to U}$ of $S$-morphisms such
that $\bigsqcup_i U_i \to U$ is fppf.

\begin{definition}
  A \cdef{category over $S$} is a pair $(\XX,\rmp)$ where $\XX$ is a category
  and $\rmp\colon \XX \rightsquigarrow \Sch_{/S}$ is a functor. A morphism $f\colon y \to z$ in
  $\XX$ is called \cdef{cartesian} if given any morphism $g\colon x \to z$ and
  a factorization $\rmp(f) \circ \phi \colon \rmp(x) \to \rmp(y) \to \rmp(z)$
  of $\rmp(g)$, there exists a unique morphism $h\colon x \to y$ such that
  $\rmp(h) = \phi$ and $g = f\circ h$.
\end{definition}

\begin{equation}
  \label{eq:cartesian-def}
  \begin{tikzcd}
    \XX \arrow[d, "\rmp"', rightsquigarrow]& & x \arrow[r, "h",
    dashed]\arrow[rr, bend left, "g"] \arrow[d, rightsquigarrow, mapsto] & y\arrow[r, "f"] \arrow[d, rightsquigarrow, mapsto] & z  \arrow[d, rightsquigarrow, mapsto]\\
    \Sch_{/S} & & \rmp(x) \arrow[r, "\phi"]& \rmp(y)\arrow[r, "\rmp(f)"] & \rmp(z)
  \end{tikzcd}
\end{equation}

\begin{definition}
  Let $(\XX,\rmp)$ be a category over $S$. If $f\colon y \to z$ is a cartesian
  morphism, the object $y \in \XX$ is called a \cdef{pullback} of $z$ along
  $\rmp(f)$. Given an $S$-scheme $U$, the \cdef{category of $U$-points in
    $\XX$}, denoted $\XX(U)$, is the category of
  pullbacks over the identity. That is,

  \medskip
  \noindent \textbf{Objects:} objects $u$ in $\XX$ such that $\rmp(u) = U$.\\
  \textbf{Morphisms:} morphisms $\phi\colon v \to u$ in $\XX$ such that
  $\rmp(\phi) = \id_U$.
\end{definition}

\begin{definition}
  A \cdef{fibered category over $S$} is a category $(\XX,\rmp)$ over $S$ such
  that for every $S$-morphism of schemes $\Phi\colon V \to U$ and $u$ in
  $\XX(U)$, there exists a cartesian morphism $\phi\colon v \to u$ such that
  $\rmp(\phi) = \Phi$. In particular, this implies that $v$ is in $\XX(V)$.
\end{definition}

Fibered categories over $S$ assemble into a $2$-category (see \cite[Definition
3.1.3]{Olsson16}). Indeed, there are natural notions of (i) morphisms between
fibered categories over $S$, and (ii) morphisms between morphisms of fibered
categories over $S$. Moreover, there is a version of the Yoneda lemma (see
\cite[Chapter 3.2]{Olsson16}) in this context that justifies calling
$U \mapsto \XX(U)$ a ``functor'' of points.

\begin{definition}
  Recall that a \cdef{groupoid} is a category in which every morphism is an
  isomorphism. A \cdef{category fibered in groupoids over $S$} is a fibered
  category $\XX$ over $S$, such that for every $S$-scheme $U$, the category
  $\XX(U)$ is a groupoid. Given a category fibered in groupoids $\XX$ over $S$
  and an $S$-scheme $U$, we denote by $\XX\ideal{U}$ the set of isomorphism
  classes of the groupoid $\XX(U)$.
\end{definition}

Since our focus will be on the arithmetic of stacks, thinking about stacks in
terms of their groupoids/sets of $U$-points will be enough for most of our
applications. When we use the word \emph{stack}, we mean an \cdef{algebraic stack} in
the sense of \cite[Section 8.1, Section 8.3]{Olsson16}.

% \begin{definition}
%     \label{def:stack} 
%     Let $\XX$ be a category fibered in groupoids over $S$.
%     \begin{enumerate}[label=(\roman*), leftmargin=*]
%     \item $\XX$ is a \cdef{stack} if for every fppf cover $\brk{U_i \to U}$, the
%       induced descent functor $\XX(U) \to \XX(\brk{U_i\to U})$ is an equivalence of
%       categories. See \cite[Section 4.2.4]{Olsson16}.
%     \item A stack $\XX$ is \cdef{algebraic} if the diagonal
%       $\Delta\colon \XX \to \XX\times_S \XX$ is representable by an algebraic
%       space, and $\XX$ admits a smooth surjection $X' \to \XX$ from an
%       $S$-scheme $X'$. The map $X'\to \XX$ is called a \cdef{smooth
%         presentation} of $\XX$. See \cite[Section 8.1]{Olsson16}.
%     \item An algebraic stack $\XX$ is \cdef{Deligne--Mumford} if $\XX$ admits an
%       {\'e}tale surjection $X' \to \XX$ from an $S$-scheme $X'$. See
%       \cite[Section 8.3]{Olsson16}.
%     \end{enumerate}
% \end{definition}

\subsection{Review of quotient stacks}
\label{sec:review-quotient-stacks}
We focus on a concrete kind of stacks that arise from groups acting on schemes.
\begin{situation}
    \label{situation:quotient-stack-general}
    We place ourselves in the following situation for the rest of
    \Cref{sec:review-quotient-stacks}.
    \begin{itemize}[leftmargin=*]
    \item Let $S$ be a fixed base scheme.
    \item Let $Z$ be a scheme over $S$.
    \item Let $G$ be an fppf $S$-group scheme.
    \item Suppose that we have $Z\times_S G \to Z$ a right action of $G$ over $S$.
    \item We abbreviate $\HH^1(S,G) = \cHH^1_\fppf(S,G)$ for the \v{C}ech
      cohomology set on the big fppf site of $S$, as in \cite[Section
      6.4.4]{Poonen2017}.
    \end{itemize}
  \end{situation}

\begin{definition}[Torsor scheme]
    \label{def:torsor-scheme}
    Let $G\to S$ be an fppf group scheme. A \cdef{right fppf $G$-torsor over
      $S$} is an $S$-scheme $T \to S$ together with a right action
    $T\times_S G \to T$ such that the following conditions hold:
    \begin{enumerate}
    \item $T \to S$ is fppf.
    \item \label{item:torsor-scheme-2}The map $T\times_S G \to T\times_S T$ defined by
      $(t,g) \mapsto (t,t\cdot g)$ is an isomorphism.
    \end{enumerate}
    A \cdef{morphism of $G$-torsors} is a $G$-equivariant morphism of
    $S$-schemes.
\end{definition}

An obvious yet important example is the \cdef{trivial $G$-torsor}. This
is the fppf scheme $G\to S$ itself, with the right $G$-action given by the
multiplication law. In fact, all $G$-torsors are locally trivial.

\begin{lemma}
    \label{lemma:trivial-G-torsor}
    Let $T \to S$ be an $S$-scheme, equipped with a $G$-action
    $T\times_S G \to T$ satisfying \Cref{item:torsor-scheme-2} in
    \Cref{def:torsor-scheme}. The following conditions are equivalent.
    \begin{enumerate}[label=(\alph*)]
        \item \label{item:equiv-1-def-torsor} $T \to S$ is fppf.
        \item \label{item:equiv-2-def-torsor} $T \to S$ is fppf locally isomorphic to the trivial $G$-torsor.
        \item \label{item:equiv-3-def-torsor} $T \to S$ admits a section fppf locally.
    \end{enumerate}
\end{lemma}

If we care about understanding group actions (i.e., quotients), we must leave
the world of schemes. For many interesting examples, the sheafification of
$U \mapsto Z(U)/G(U)$ is not representable by a scheme. Quotient stacks
elegantly and succinctly solve this problem in terms of torsors.

\begin{definition}[Quotient stack]
  \label{def:quotient-stack}
  Define the \cdef{quotient stack} of $Z$ by $G$, denoted $[Z/G]$, to be the
  algebraic stack over $S_\fppf$ with: \\
  \textbf{Objects:} triples $(U,T,\phi)$
    \begin{equation*}
      \begin{tikzcd}
          T \arrow[d, "G_U\text{-torsor}"'] \arrow[rr, "\phi"', "G\text{-equivariant}"] & & Z \arrow[rdd, bend left, gray]& &\\
          U \arrow[rrrd, bend right, gray] & & & \\
          & & & {\color{gray}S}
            \end{tikzcd}
      \end{equation*}
    where
    \begin{enumerate}[label=(\roman*)]
    \item $U$ is an $S$-scheme,
    \item $T\to U$ is a right fppf $G_U$-torsor, and
    \item $\phi\colon T \to Z$ is a $G$-equivariant $S$-morphism.
    \end{enumerate}
    \textbf{Morphisms:} $(U',T',\phi') \to (U, T, \phi)$ are pairs $(f,h)$,
    where
    \begin{enumerate}[resume, label=(\roman*)]
    \item $f\colon U' \to U$ is an $S$-morphism of schemes, and
    \item $h\colon T' \to T$ is a $G$-equivariant morphism over $f$ inducing an
      isomorphism of $G_{U'}$-torsors $T' \cong T \times_{f,U} U'$, such that
      $\phi' = \phi\circ h$.
    \end{enumerate}
  \end{definition}

\begin{equation}
  \begin{tikzcd}
      T' \ar[d] \ar[rr, "h", color=caribbeangreen] & &  T \ar[d] \ar[rrd, "\phi"] \\
      U' \ar[rrrrd, gray] \ar[rr, "f", color=caribbeangreen] & & U \ar[rrd, gray] & & Z
      \ar[d, gray] \ar[from=ullll, "\phi'"', crossing over, near end] \\ & & &
      & {\color{gray}S}
  \end{tikzcd}
\end{equation}

In particular, for any given $S$-scheme $U$, the groupoid $[Z/G](U)$ consists
of pairs $(T,\phi)$ with $T\to U$ a $G_U$-torsor, and $\phi\colon T \to Z$ a
$G$-equivariant $S$-morphism; and isomorphisms
$h\colon (T_1,\phi_1) \to (T_2,\phi_2)$ are simply isomorphisms
$h\colon T_1 \to T_2$ of $G_U$-torsors, compatible with the maps to $Z$.
\begin{equation}
  \begin{tikzcd}
      T_1 \ar[rd] \ar[rr, "h", color=caribbeangreen] & &  T_2 \ar[ld] \ar[rrd, "\phi_2"] \\
      & U \ar[rrrd, gray] & & &  Z \ar[from=llllu, "\phi_1"', near end, crossing over] \ar[d, gray] \\
      & & & & {\color{gray}S}
   \end{tikzcd}
 \end{equation}

\subsection{Descent theory revisited}
\label{sec:descent-theory-revisited}
In this section, we summarize the basics of descent theory from the point of
view of quotient stacks. We follow Skorobogatov's book \cite[Section
2.2]{Skorobogatov01}, but with inverted handedness. We place ourselves in the
following situation until the end of \Cref{sec:descent-theory-revisited}.

\begin{situation}
\label{situation:geometric-operations-on-torsors}
Suppose we are in \Cref{situation:quotient-stack-general}. 
\begin{itemize}[leftmargin=*]
\item Assume that $Z$ is quasi-projective.
\item Let $T$ denote an $S$-scheme with a \textbf{left} $G$ action over $S$.
\item Let $q\colon Z \to [Z/G]$ be the natural projection map.
\item The structure map $G \to S$ is affine, and $S$ is locally noetherian.
\end{itemize}    
\end{situation}

\begin{remark}
  The last assumption ensures that there is a bijective correspondence between
  $\HH^1(S,G)$ and isomorphism classes of $G$-torsor schemes, as a consequence
  of \cite[Theorem 6.5.10]{Poonen2017}.
\end{remark}

\begin{remark}
  It is possible to work in greater generality (see \cite[Section
  2.5]{Santens2023}) if we are not concerned with representability. In general,
  the contracted product of an $S$-scheme $Z$ with a left $G$-torsor will be an
  algebraic space.
\end{remark}

Recall that when $G$ is not commutative, $\HH^1(S,G)$ is only a pointed set
(the distinguished element corresponds to the class of the trivial $G$-torsor)
and not an abelian group. Nevertheless, we can still perform certain algebraic
operations in this pointed set in terms of corresponding geometric operations
on torsors.
\begin{definition}[Contracted product]
  \label{def:contracted-product}
  The \cdef{contracted product} $Z\conprod{G} T$ is defined as the quotient
  stack $[Z\times_S T / G]$, where $G$ acts on the \textbf{right} on
  $Z\times_S T$ via $$(z,t)\cdot g \colonequals (z\cdot g, g^{-1}\cdot t).$$
\end{definition}

A crucial application of this definition is the pushforward operation on
torsors. Given $\varphi\colon G \to H$ a homomorphism of fppf group schemes over
$S$, we consider the left action of $G$ on $H$. If $P \to S$ is a $G$-torsor, the
contracted product
\[\varphi_*P \colonequals P \conprod{G} H, \quad \text{ where } (p,h)\cdot g \colonequals (p\cdot g,
  \varphi(g)^{-1}\cdot h), \] turns out to be an $H$-torsor, called the
\cdef{pushforward of $P$ by $\varphi$}. As an application of this construction, we
have the following lemma (see \cite[Exercise
10.F]{Olsson16}).
\begin{lemma}[Induced maps on quotient stacks]
    \label{lemma:induced-map-quotient-stacks}
    Let $S$ be a scheme, and $\varphi\colon G \to H$ a homomorphism of fppf
    group schemes over $S$. Let $X$ be an $S$-scheme with a right $G$-action,
    and $Y$ an $S$-scheme with a right $H$-action. Suppose that there is an
    $S$-morphism $f\colon X \to Y$ that is compatible with the group actions.
    Then, $f$ induces a morphism of algebraic stacks
    $\bar{f}\colon [X/G] \to [Y/H]$.
\end{lemma}
\begin{proof}[Proof sketch]
  Let $U \to S$ be an $S$-scheme. At the level of $U$-points, the functor
  $\bar f(U) \colon [X/G](U) \to [Y/H](U)$ is defined in the following way.
  Recall that a $U$-point on $[X/G]$ is a triple $(U, T, \phi)$ as in
  \Cref{def:quotient-stack}. First, consider the pullback
  $\varphi_U\colon G_U \to H_U$. Pushing forward the $G_U$-torsor $T$ via
  $\varphi_U$ gives a triple $(U, (\varphi_U)_*T, (\varphi_U)_*(f\circ \phi))$.
  \begin{equation*}
    \begin{tikzcd}
      T \arrow[d] \arrow[r, "\phi"] &  X \arrow[d, gray, ""{name=L, below}]
      & (\varphi_U)_*T \arrow[r] \arrow[d, ""{name=R}] & Y \arrow[d, gray] & \\
      U \arrow[r, gray] & {\color{gray} [X/G]} & U \arrow[r, gray] &
      {\color{gray} [Y/H]} \arrow[mapsto, from=L, to=R,
      caribbeangreen, shorten=7mm]
     \end{tikzcd}
  \end{equation*}
\end{proof}

The following lemma is a restatement of \cite[Section 6.5.6]{Poonen2017}.
\begin{lemma}[Twisting by fppf descent]
    \label{lemma:twisting}
    Given $\tau \in \HH^1(S, G)$, let $T \to S$ be a \textbf{left} fppf
    $G$-torsor corresponding to $\tau$. Then:
    \begin{enumerate}[label=(\roman*), leftmargin=*]
    \item \label{item:Z-tau} The contracted product $Z\conprod{G} T$ is represented by a
      quasi-projective $S$-scheme $Z_\tau$. We call this the \cdef{twist of $Z$ by
        $\tau$}.
    \item \label{item:trivial-torsor} If $T = G$ is the trivial left
      $G$-torsor, then $Z_\tau \cong Z$ as $S$-schemes with a right $G$-action.
    \item \label{item:inner-twist} Taking $Z = G$ acting on itself by conjugation, the twist
      $Z_\tau = G_\tau$ is an affine fppf group scheme over $S$. It is called the
      \cdef{inner twist} of $G$ by $\tau$.
    \item \label{item:right-Gt-action} The twist $Z_\tau$ is an $S$-scheme with
      a right $G_\tau$-action, and there is an isomorphism
      $[Z/G] \cong [Z_\tau/G_\tau]$. In particular, there is an induced map
      $q_\tau\colon Z_\tau \to [Z/G]$, called the \cdef{twist of $q$ by
        $\tau$}. When $Z$ is itself a right fppf $G$-torsor over $S$, the twist
      $Z_\tau$ is a right fppf $G_\tau$-torsor over $S$.
    \item \label{item:bitorsor} The $S$-scheme $T$ is a $(G,G_\tau)$-bitorsor. The
      same $S$-scheme with the inverse $G$ action
      $t\cdot g \colonequals g^{-1}\cdot t$ is a $(G_\tau,G)$-bitorsor. We denote the
      \cdef{inverse (right) $G$-torsor} by $T^{-1}$.
    \item \label{item:inverse-torsor} Finally, the contracted product
      $T\inv\conprod{G}T$ is isomorphic to the trivial $G$-torsor.
    \end{enumerate}
\end{lemma}

\begin{proof}
  \ref{item:Z-tau} The representability of $Z_\tau = Z\conprod{G}T$ is an
  application of fppf descent. See \cite[Lemma 2.2.3]{Skorobogatov01} for a
  proof when $Z$ is affine. % The general case follows by Thm 6 in Section 6.1
                            % of the Neron models book. See the paragraph
                            % before Example C in page 141.

  \ref{item:trivial-torsor} We have the morphism $Z\times_S G \to Z\times_S Z$ given by
  $(z,g) \mapsto (z,z\cdot g)$. Observe that it is $G$-equivariant for the twisted action
  on $Z\times_S G$ defined in \Cref{def:contracted-product}, and the action
  $(z_1,z_2)\cdot g \colonequals (z_1\cdot g,z_2)$ on $Z\times_S
  Z$. % (z,g).h = (z.h,
                                % h^-1g) |-> (z.h, (z.h).h^-1g) = (z.h, z.g)
                                % <-| (z,z.g).h
  This gives a morphism
  of quotient stacks $Z_\tau = [Z\times_S G/G] \to [Z\times_S Z / G]$. Since the first
  projection $Z\times_S Z \to Z$ is $G$-equivariant for the trivial $G$-action
  on $Z$, we get a map $\psi\colon Z_\tau \to Z$. On the other hand, we have a morphism
  $\phi\colon Z \to Z_\tau$ induced by $Z \to Z\times_S G$. To see that these are mutual inverses,
  it is enough to realize that the following diagram is commutative
  \begin{equation*}
    \begin{tikzcd}
      Z \arrow[d] \arrow[dr, bend left, "\phi"] & \\
      Z\times_S G \arrow[r] \arrow[d] & Z_\tau \arrow[d, "\psi"] \\
      Z\times_S Z \arrow[r, "\mathrm{pr}_1"'] & Z .
    \end{tikzcd}
  \end{equation*}
  
  \ref{item:inner-twist} Since $G \to S$ is affine, the same will be true for
  $G_\tau\to S$ by fppf descent. Checking that $G_\tau$ is an $S$-group is a matter of
  pulling back the group operations to $T$ and verifying that they are
  $G$-equivariant under the twisted action (\Cref{def:contracted-product}). For
  example, consider the inverse morphism $\iota\colon G \to G$ pulled back to
  $\iota\times_S T\colon G\times_S T \to G\times_S T$. Then, we have that
  $(g,t)\cdot h = (h^{-1}gh, h^{-1}t)$ maps to
  $(h^{-1}g^{-1}h, h^{-1}t) = (g^{-1}, t)\cdot h$. We obtain the twisted inverse
  morphism $G_\tau \to G_\tau$ by passing to the quotient.

  \ref{item:right-Gt-action} Consider the morphism
  $\phi\colon (Z\times_S T)\times_S (G\times_S T) \to Z\times_S T$ given on points by
  $(z,x,g,t) \mapsto (z\cdot g, t)$. Note that
  $(z,x,g,t)\cdot h = (z\cdot h, h^{-1}x,h^{-1}gh,h^{-1}t)$ maps to
  $(z\cdot gh, h^{-1}t) = (z\cdot g, t)\cdot h$, so $\phi$ induces a morphism
  $Z_\tau\times_S G_\tau \to Z_\tau$. One similarly verifies the
  $G$-equivariance of the diagrams that descend to the group action axioms on
  $Z_\tau \times_S G_\tau \to Z_\tau$. For the third statement, note that we have a morphism
  $Z\times_S T \to Z$ compatible with $\varphi\colon G_\tau \to G$, namely the first projection
  $(z,t) \mapsto z$. From \Cref{lemma:induced-map-quotient-stacks}, we get that the
  induced map of quotient stacks $[Z_\tau/G_\tau] \to [Z/G]$ is in fact an
  isomorphism. We obtain $q_\tau$ as the composition $Z_\tau \to [Z_\tau/G_\tau] \cong [Z/G]$.

  \ref{item:bitorsor} Follows directly from \ref{item:right-Gt-action}.

  \ref{item:inverse-torsor} Using \Cref{lemma:trivial-G-torsor}, we exhibit a
  global section of $T\inv\conprod{G}T \to S$. The diagonal
  $\Delta\colon T \to T\times_S T$, $t \mapsto (t,t)$, is equivariant for the action
  $(t_1,t_2)\cdot g \colonequals (g^{-1}\cdot t_1, g^{-1}\cdot t_2)$ defining the
  contracted product, since
  $\Delta(g^{-1}\cdot t) = (g^{-1}\cdot t, g^{-1}\cdot t)$. It therefore descends to a morphism
  $$S = [T/G] \to T\inv\conprod{G}T,$$ which is a section of the structure map.
  Hence $T\inv\conprod{G}T$ is the trivial $G$-torsor.  
\end{proof}

In general, the sets $[Z/G]\ideal{S}$ and $Z(S)/G(S)$ are not the same.
Nevertheless, $Z(S)/G(S)$ is contained in $[Z/G]\ideal{S}$, and the difference
is accounted for by the quotients $Z_\tau(S)/G(S)$, as $\tau$ ranges over
$\HH^1(S, G)$.

\begin{theorem}[Descent theory partition]
  \label{thm:descent}
  The \textbf{set} of $S$-points on the quotient stack $[Z/G]$ is
  partitioned by the images of the $S$-points of the twists of
  $q\colon Z \to [Z/G]$.
  \begin{equation*} [Z/G]\ideal{S} = \bigsqcup_{\tau\in \HH^1(S,G)} q_\tau(Z_\tau(S)).
  \end{equation*}
\end{theorem}
\begin{proof}
  Recall that a map $S \to [Z/G]$ is the data of a triple $(S,T^{-1},\phi)$
  where $T^{-1}$ is a right fppf $G$-torsor over $S$, and
  $\phi\colon T^{-1} \to Z$ is a $G$-equivariant map of $S$-schemes. We want to
  show that every map $(T^{-1},\phi)\colon S \to [Z/G]$ factors through a twist
  $q_\tau\colon Z_\tau \to [Z/G]$ of the canonical quotient
  $q\colon Z \to [Z/G]$, where $\tau$ is completely determined by the
  isomorphism class of the point $(T,\phi)$. Indeed, in this setting, we have
  the \emph{evaluation map}
  $\zeta\colon (T^{-1},\phi) \mapsto \tau \colonequals [T\to S]$ from
  $[Z/G]\ideal{S}$ to $\HH^1(S,G)$, where $\tau$ is the cohomology class
  corresponding to the left $G$-torsor $T \to S$. Since $T^{-1}\conprod{G} T$
  is isomorphic to the trivial $G$-torsor, we have a section
  $e\colon S \to G \cong T^{-1}\conprod{G} T$ that realizes the factorization of our
  map $(T^{-1},\phi)$ by the commutativity of the diagram in
  \Cref{fig:descent}. The map $T^{-1}\conprod{G}T \to Z_\tau$ is the one
  induced by the $G$-equivariant $S$-morphism
  $\phi\times_S \id_T\colon T^{-1}\times_S T \to Z\times_S T$.
  \begin{figure}[ht]
    \centering
    \begin{tikzcd}
      T^{-1}\conprod{G}T \arrow[rr,"\mathrm{\phi\conprod{G} id_T}"] & & Z_\tau \arrow[dd, gray, "q_\tau"]    \\
      T^{-1} \arrow[d] \arrow[r, "\phi"] & Z \arrow[d, gray, "q"] &  \\
      S \arrow[uu, dashed, bend left = 50, "e"] \arrow[r, gray] & {\color{gray}
        [Z/G]} \arrow[r, equal]& {\color{gray} [Z_\tau/G_\tau]}
        \end{tikzcd}
        \caption{Proof of the method of descent.}
        \label{fig:descent}
      \end{figure}
    \end{proof}

    As a reality check, let us calculate the set of $R$-points on the
    projective line over a principal ideal domain $R$ using \Cref{thm:descent}.
    \begin{example}[PID points on the projective line]
      \label{ex:PID-points-P1} Recall the greatest common divisor of two
      elements $s,t$ in $R$ is a generator of the ideal $sR + tR$. Let
      $\mcV \colonequals \A^2-\mathbf{0}$, so that
      $\mcV(R) \cong \brk{(s,t)\in R^2 : sR + tR = R}$. We have that
      $\Pone(R) \cong \brk{(s,t) \in R^2 : sR + tR = R}/R^\times$. One can see
      this using the fact that $\Pone$ is the quotient stack $[\mcV / \Gm]$.
      Indeed, since $\Pic R$ is trivial,
      \[\Pone(R) = [\mcV/\Gm] \ideal{R} = \bigsqcup_{\tau\in\HH^1(R,\Gm)}
    \mcV_\tau(R)/\Gm(R) = \mcV(R)/R^\times.\] 

  Indeed, a point $Q \in \Pone(R)$ is (isomorphic to a) cartesian square
  \begin{equation*}
    \begin{tikzcd} \Gm \arrow[r, "\phi"] \arrow[d] & \mcV
      \arrow[d] \\ \Spec R \arrow[r] & \Pone,
        \end{tikzcd}
      \end{equation*}
      where $\phi$ is a $\Gm$-equivariant map. Composing the identity section
      $e\colon \Spec R \to \Gm$ with $\phi$ we obtain a point in $\mcV(R)$,
      i.e., a pair $(s,t) \in R^2$ such that $sR + tR = R$. Any other
      isomorphic square comes from a $\Gm$-equivariant map
      $\phi'\colon \Gm \to \mcV$ giving rise to a point $(s',t')$ such that
      $(s',t') = (us, ut)$ for some $u \in R^\times$.
\end{example}

\section{Root stacks and the Belyi stack}
\label{sec:root-stacks-belyi}

\subsection{Review of the root stack construction}
\label{sec:review-root-stack}
An \cdef{effective Cartier divisor} on a scheme $X$ is a closed
subscheme $D \subset X$ such that the corresponding ideal sheaf \(\OO_X(-D)\)
is a line bundle \spcite{01WR}. Equivalently, a closed subscheme is an
effective Cartier divisor if and only if it is locally cut out by a single
element which is a non-zero divisor \spcite{01WS}. Denote by $j_D\colon \OO_X(-D)
\hookrightarrow \OO_X$ the natural inclusion morphism of $\OO_X$-modules.

\begin{definition}[{\cite[Definition 10.3.2]{Olsson16}}]
  A \cdef{generalized effective Car-tier divisor} on a scheme $X$ is a pair
  $(\mathcal{L},\rho)$, where $\mathcal{L}$ is a line bundle on $X$, and
  $\rho\colon \mathcal{L} \to \OO_X$ is a morphism of $\OO_X$-modules. An
  \cdef{isomorphism of generalized Cartier divisors}
  $(\mathcal{L}', \rho') \cong (\mathcal{L},\rho)$ is an isomorphism of line
  bundles $\sigma\colon \mathcal{L}' \to \mathcal{L}$ such that the following
  triangle commutes
  \begin{equation*}
      \begin{tikzcd}
        \mathcal{L}' \ar[rr, "\sigma"] \ar[dr,"\rho'"'] & & \mathcal{L} \ar[dl, "\rho"]\\
        & \OO_X & .
      \end{tikzcd}
    \end{equation*}
    We can multiply generalized effective Cartier divisors $(\mathcal{L},\rho)$
    and $(\mathcal{L}',\rho')$ by declaring
    $(\mathcal{L},\rho)\cdot (\mathcal{L}',\rho') \colonequals
    (\mathcal{L}\otimes_{\OO_X}\mathcal{L}',\rho\otimes\rho')$, where
    $\rho\otimes \rho'$ is the morphism of $\OO_X$-modules given by the
    composition
    \[\mathcal{L}\otimes_{\OO_X}\mathcal{L}' \to \OO_X \otimes_{\OO_X}\OO_X
    \cong \OO_X.\]
\end{definition}

\begin{example}[Effective Cartier divisors]
  \label{example:effective-cartier-divisors}
  Given an effective Cartier divisor $D \subset X$, the pair $(\OO_X(-D), j_D)$
  is a generalized effective Cartier divisor. By definition, two effective
  Cartier divisors $D', D \subset X$ are isomorphic as generalized effective
  Cartier divisors if and only if they are equal and the isomorphism is
  therefore unique.
\end{example}

\begin{example}[Generalized effective Cartier divisors on affine schemes]
    \label{ex:GECD-affine-schemes}
    In light of the equivalence between $R$-modules and quasicoherent
    $\OO_X$-modules on $X = \Spec R$, a generalized effective Cartier divisor
    on an affine scheme is of the form $(\widetilde{M},\widetilde{\lambda})$
    for a projective $R$-module $M$ of rank one, and a morphism
    $\lambda\colon M \to R$ of $R$-modules. In particular, $\lambda(M)$ is an ideal
    in $R$. Two generalized effective Cartier divisors $(M',\lambda')$ and
    $(M,\lambda)$ on $\Spec R$ are isomorphic if and only if there exists an
    $R$-module isomorphism $\sigma\colon M' \to M$ such that
    $\lambda' = \lambda\circ \sigma$. In particular, note that such a pair
    gives rise to the same ideal
    $\lambda'(M') = \lambda(\sigma(M')) = \lambda(M)$.
  \end{example}

\begin{definition}[Root stack]
  \label{def:root-stack}
  Fix an effective Cartier divisor $D$ on a scheme
  $X$, and a positive integer $r$. Let $\sqrt[r]{X;D}$ be the
  fibered category over $X_\fppf$ with:

  \medskip
  \noindent \textbf{Objects:} triples $(f\colon T \to X, (\mathcal{M},\lambda), \sigma)$
  where $f\colon T \to X$ is an $X$-scheme, $(\mathcal{M},\lambda)$ is a
  generalized effective Cartier divisor on $T$, and $\sigma\colon
  (\mathcal{M}^{\otimes r},\lambda^{\otimes r}) \to (f^*\OO_X(-D),f^*j_D)$
  is an isomorphism of generalized effective Cartier divisors on $T$.

  \medskip
  \noindent \textbf{Morphisms:} a morphism
  \[(f'\colon T' \to X, (\mathcal{M}',\lambda'), \sigma') \to (f\colon T \to X,
  (\mathcal{M},\lambda), \sigma)\] is the data of a pair $(h, h^\flat)$ where
  $h\colon T' \to T$ is an $X$-morphism, and
  $h^\flat\colon (\mathcal{M}',\lambda') \to (h^*\mathcal{M}, h^*\lambda)$ is
  an isomorphism of generalized effective Cartier divisors on $T'$ such that
  the following diagram commutes
  \begin{equation*}
    \begin{tikzcd}
      \mathcal{M}'^{\otimes r} \ar[r,"h^{\flat \otimes r}"] \ar[d, "\sigma'"'] & h^*\mathcal{M}^{\otimes r} \ar[d,"h^*\sigma"] \\
      (f')^*\OO_X(-D) \ar[r,"\sim"] & h^*f^*\OO_X(-D).
    \end{tikzcd}
  \end{equation*}
\end{definition}

\begin{remark}[Points on a root stack]
   \label{remark:points-on-root-stack}
   Usually, the base scheme $X$ is itself defined over a different base scheme
   $S$. If $\XX = \sqrt[r]{X;D}$, it is common to abuse notation and to write
   $\XX(S)$ for $\XX(f)$ when the structure morphism $f$ is clear from context.
   What we mean is that we are considering $\XX$ as a stack over $S$ via the
   forgetful map $X_\fppf \to S_\fppf$. In particular, it follows that the
   groupoid $\XX(S)$ is the disjoint union over
   $x \in \mathrm{Hom}_S(S, X) = X(S)$ of the groupoids $\XX(x)$.
\end{remark}

\begin{remark}[Fiber groupoid]
  Using the notation of \Cref{def:root-stack}, and given an $X$-scheme
  $f\colon T\to X$, we see that the fiber groupoid $\sqrt[r]{X;D}(f)$ consists
  of:

  \medskip
  \noindent \textbf{Objects:} triples
  $(f\colon T \to X ,(\mathcal{M},\lambda), \sigma)$ where
  $(\mathcal{M},\lambda)$ is a generalized effective Cartier divisor on $T$,
  and
  $\sigma\colon (\mathcal{M}^{\otimes r},\lambda^{\otimes r}) \to
  (f^*\OO_X(-D),f^*j_D)$ is an isomorphism of generalized effective Cartier
  divisors on $T$.

  \medskip
  \noindent\textbf{Isomorphisms:}
  $(f\colon T \to X, (\mathcal{M}',\lambda'), \sigma') \to (f\colon T \to X,
  (\mathcal{M},\lambda), \sigma)$ consist of pairs $(h, h^\flat)$ where
  $h \in \Aut(T)$ satisfies $f = f\circ h$, and
  $h^\flat\colon (\mathcal{M}',\lambda') \to (h^*\mathcal{M}, h^*\lambda)$ is
  an isomorphism of generalized effective Cartier divisors on $T$ such that the
  following diagram commutes
\begin{equation*}
  \begin{tikzcd} \mathcal{M}'^{\otimes r} \ar[r,"h^{\flat \otimes r}"] \ar[d,
"\sigma'"'] & h^*\mathcal{M}^{\otimes r} \ar[d,"h^*\sigma"] \\ (f)^*\OO_X(-D)
\ar[r,"\sim"] & h^*f^*\OO_X(-D).
  \end{tikzcd}
\end{equation*}
\end{remark}

Finally, we arrive at the main definition of this section.
\begin{definition}[Iterated root stack]
  \label{def:iterated-root-stack}
  Let $X$ be a scheme. Take a finite list $D_1, \dots, D_r$ of effective
  Cartier divisors on $X$, and let $n_1, \dots, n_r$ be positive integers. The
  \cdef{iterated root stack} of $X$ at the divisors $D_1, \dots, D_r$ with
  multiplicities $n_1, \dots, n_r$ is the fiber product
  \begin{equation}
    \label{eq:iterated-root-stack}
    \sqrt[\leftroot{2} \uproot{5} n_1]{X;D_1} \times_X \cdots \times_X
    \sqrt[\leftroot{2} \uproot{5} n_r]{X;D_r} \to X.
  \end{equation}
\end{definition}

\subsection{The projective line rooted at a point}
\label{sec:proj-line-rooted-at-pt}
Our first concrete non-trivial example of a root stack is the projective line
rooted at a single point $\XX \colonequals \sqrt[n]{\Pone; P}$.

\begin{definition}
  \label{def:IPQ}
  Let $R$ be a principal ideal domain, and choose
  $P = (c:d)$ and $Q = (a:b)$ in $\Pone(R)$. Define the \cdef{intersection
    ideal of $P$ with $Q$} as  $I(P,Q) \colonequals (ad-bc)R \subset R$.
\end{definition}

The ideal $I(P,Q)$ cuts out the locus in $\Spec R$ over which $P$ and $Q$
intersect. Indeed, the pullback of the diagonal $\Pone \to \Pone\times\Pone$ by
$(P,Q)\colon \Spec R \to \Pone\times\Pone$ gives the closed subscheme
$\Spec R/I(P,Q)$. From the magic square, $I(P,Q)$ can equivalently be defined
by the cartesian square
\begin{equation}
  \label{eq:I(P,Q)-cartesian}
  \begin{tikzcd}
    \Spec R/I(P,Q) \arrow[r] \arrow[d] & \Spec R \arrow[d, hook, "Q"]\\
    \Spec R \arrow[r, hook, "P"'] & \Pone_R \, .
  \end{tikzcd}
\end{equation}

\begin{warning}
  \label{warning:pullback-ideal-sheaf}
  The pullback $Q^*\OO_{\Pone}(-P)$ does not coincide with the sheaf
  corresponding to $I(P,Q)$. More generally, the pullback of a quasicoherent
  ideal sheaf need not coincide with the ideal sheaf of the pulled back closed
  subscheme. %see \cite[Remark 14.5.10]{Vakil23}).
  Nevertheless, we have the following commutative diagram of sheaves on
  $\Spec R$ with exact rows
  \begin{equation}
    \label{eq:diagarm-P^*O(-Q)}
    \begin{tikzcd}
      & Q^*\OO_{\Pone}(-P) \arrow[r] \arrow[d] \arrow[dr, "\widetilde\lambda"] & Q^*\OO_{\Pone} \arrow[r]
      \arrow[d, equals] &
      Q^*P_*\widetilde R \arrow[r] \arrow[d, equals] & 0 \\
      0 \arrow[r] & \widetilde{I(P,Q)} \arrow[r] & \widetilde R \arrow[r] &
      \widetilde{R/I(P,Q)} \arrow[r] & 0 \, .
    \end{tikzcd}
  \end{equation}
\end{warning}

\begin{proposition}
    \label{prop:P1-rooted-at-P}
    Let $R$ be a principal ideal domain with fraction field $K$. Let
    $\Pone = \Proj R[\s,\t]$. Fix a point $P \in \Pone(R)$, and a positive
    integer $n$. Let $\XX \colonequals \sqrt[n]{\Pone; P}$ be the $n^\th$ root
    stack of $\Pone$ at $P$, defined over $\Spec R$. Then,
    $$\XX(R) = \bigsqcup_{Q \in \Pone(R)} \XX(Q),$$ where
    \begin{enumerate}[label=(\roman*), leftmargin=*]
    \item The fiber $\XX(P)$ has automorphism groups isomorphic to
      $\mu_n(R) = \brk{u \in R^\times : u^n = 1}$, and the set of isomorphism classes
      $\XX\langle P \rangle$ is in bijection with $R^\times/(R^\times)^n$.
    \item For $Q \neq P$ the ideal $I(P,Q)$ is non-zero, and the fiber $\XX(Q)$
      contains one object with trivial automorphism group if and only if
      $I(P,Q) = J^n$ for some non-zero ideal $J \subseteq R$, and is empty otherwise.
    \end{enumerate}
    In particular, when $R = K$, we have that
    \[\XX\ideal{K} \;\cong\; \big(\Pone(K)\smallsetminus\brk{P}\big) \;\sqcup\;
      K^\times/(K^\times)^n.\]
\end{proposition}

\begin{proof}
  Let $\XX \colonequals \sqrt[n]{\Pone; P}$. As explained in
  \Cref{remark:points-on-root-stack}, the groupoid $\XX(R)$ is the disjoint
  union of the groupoids $\XX(Q)$, ranging over $Q\in \Pone(R)$. We proceed to
  describe each groupoid $\XX(Q)$.

  To start, consider the pullback of the ideal sheaf
  $\OO_{\Pone}(-P) = \widetilde{I_P}$ via the map $Q\colon \Spec R \to \Pone$,
  where $I_P = (c\,\t-d\,\s)R[\s,\t] \subset R[\s,\t]$. This is a line bundle on
  $\Spec R$ corresponding to a certain free $R$-module of rank one $M(P,Q)$.
  % (Explicitly, $M(P,Q)$ is the
  % degree zero part of the tensor product of graded $R$-modules
  % % https://stacks.math.columbia.edu/tag/09LL,
  % % https://stacks.math.columbia.edu/tag/01O5,
  % % and also EGA2 page 20.
  % \begin{align*}
  %   (a\t-b\s)R[\s,\t]\otimes_{R[\s,\t]}\dfrac{R[\s,\t]}{(c\t-d\s)R[\s,\t]},
  % \end{align*}
  % but we will not use this description.)
  Moreover, the pullback of the generalized effective Cartier divisor
  $j_P\colon \OO_{\Pone}(-P) \hookrightarrow \OO_\Pone$ corresponds to an
  $R$-module homomorphism $\lambda(P,Q)\colon M(P,Q) \to R$ with image
  $I(P,Q)$, as illustrated in Diagram~\ref{eq:diagarm-P^*O(-Q)}.
  
  The \textbf{objects} in $\XX(Q)$ are triples $(Q, (M,\lambda), \sigma)$,
  where
    \begin{itemize}[leftmargin=*]
    \item $(M,\lambda)$ is a generalized effective Cartier divisor on $\Spec R$
      (see \Cref{ex:GECD-affine-schemes}). Since $R$ is a principal ideal
      domain, $M$ is a free $R$-module of rank one and $\lambda\colon M \to R$
      is an $R$-module homomorphism.
    \item $\sigma\colon (M^{\otimes n}, \lambda^{\otimes n}) \to (M(P,Q), \lambda(P,Q))$ is an isomorphism of generalized effective Cartier
      divisors on $\Spec R$, that is, a commutative triangle of $R$-modules
      \begin{equation}
      \label{eq:XX(Q)-point}
      \begin{tikzcd}
        M^{\otimes n} \arrow[rr, "\sigma"', "\cong"] \arrow[rd,
        "\lambda^{\otimes n}"'] & & M(P,Q) \arrow[dl,"{\lambda(P,Q)}"] \\
        & R .&
      \end{tikzcd}
    \end{equation}
  \end{itemize}
  By definition, an \textbf{isomorphism}
  $(Q,(M',\lambda'),\sigma') \to (Q, (M,\lambda),\sigma)$ in $\XX(Q)$ is a pair
  $(h, h^\flat)$, where
  \begin{itemize}[leftmargin=*]
    \item $h\colon \Spec R \to \Spec R$ is a morphism over $\Spec R$, so it must be the
      identity.

    \item $h^\flat\colon M' \to M$ is an isomorphism of
      $R$-modules such that $\lambda' = \lambda\circ h^\flat$ and the
      following diagram commutes
    \begin{equation}
    \label{eq:iso-root-stacks}
        \begin{tikzcd}
        M'^{\otimes n} \ar[rd] \ar[rr, "h^{\flat \otimes n}", color=caribbeangreen] & &  M^{\otimes n}\ar[ld] \ar[rrd, "\sigma"] \\
        & R^{\otimes n} \ar[rrrd, gray, "\cong"] & & &  M(P,Q) \ar[from=llllu,
        "\sigma'"', near end, crossing over] \ar[d, gray, "{\lambda(P,Q)}"] \\
        & & & & {\color{gray}R.}
        \end{tikzcd}
      \end{equation}
    \end{itemize}

    \begin{enumerate}[label=(\roman*), leftmargin=*]
    \item When $P = Q$, then $I(P,Q) = 0$ and this forces every map
      $\lambda\colon M \to R$ to be the zero map. In particular, the bottom part of
      diagram (\ref{eq:iso-root-stacks}) imposes no restriction: trivialising
      $M \cong R$ and $M(P,P) \cong R$, the datum $\sigma$ becomes an element of
      $R^\times$ and $h^\flat = u \in R^\times$ acts by
      $\sigma \mapsto u^n \sigma$. The isomorphisms of $\XX(P)$ are precisely the isomorphisms
      of $R$-modules $h^\flat\colon M' \to M$ such that
    \begin{equation*}
      \begin{tikzcd}
        (M')^{\otimes n} \arrow[rr, "h^{\flat \otimes n}"', "\cong"] \arrow[rd, "\sigma'"'] & & M^{\otimes n} \arrow[ld, "\sigma"] \\
        & M(P,P). &
      \end{tikzcd}
    \end{equation*}
    In particular, any triple $(P,(M,\lambda), \sigma)$ in $\XX(P)$ has
    $\mu_n(R)$ automorphisms.

  \item When $P \neq Q$, the commutativity of (\ref{eq:XX(Q)-point}) requires that
    the non-zero ideal $I(P,Q)$ is the $n^\th$ power of the ideal $\lambda(M)$ in
    $R$. This condition is also sufficient. Indeed, if $I(P,Q) = J^n$ for some
    non-zero ideal $J \subseteq R$, with inclusion
    $\lambda\colon J \hookrightarrow R$, then take an isomorphism of
    $R$-modules $\sigma\colon I(P,Q) \to M(P,Q)$ and note that
    \begin{equation}
        \label{eq:Q-canonical-triple}
        (Q, (J, \lambda), \sigma\colon J^n \to M(P,Q))
    \end{equation}
    is an object of $\XX(Q)$, and every object in $\XX(Q)$ is isomorphic to it.
    To calculate the automorphism group of this object, note that the only
    possible isomorphism $h^\flat\colon J \to J$ of $R$-modules such that
    $\lambda = \lambda \circ h^\flat \colon J\hookrightarrow R$, is the
    identity. Thus, the automorphism groups in $\XX(Q)$ are trivial.
\end{enumerate}
\end{proof}

\subsection{The Belyi stack}
\label{sec:belyi-stack}
In this section, we summarize a few geometric and arithmetic properties of the
\belyi stack $\Pone\abc$. This is the stack corresponding to Darmon's $M$-curve
$\mathbf{P}^1_{a,b,c}$ in \cite[p.~4]{Darmon97}.

\begin{situation}
\label{situation:belyi-stack}
    Let
    \begin{itemize}[leftmargin=*]
    \item $(a,b,c) \in \Z^3$ be a triple of positive integers,
    \item $\Pone = \Proj \Z[\s,\t]$, and
    \item $P_0 = V(\s), P_1 = V(\s-\t), P_\infty = V(\t) \in \Div(\Pone_\Z)$.
    \end{itemize}
\end{situation}

\begin{definition}[\belyi stack]
  \label{def:belyi-stack}
  We define the \cdef{\belyi stack} $\Pone\abc$ as the iterated root stack of
  $\Pone_\Z$ at the divisors $P_0, P_1, P_\infty$ with multiplicities $a,b,c$.
  \begin{equation*}
  \Pone\abc \colonequals \left(\sqrt[\leftroot{2} \uproot{5} a]{\Pone; P_0}\right) \times_{\Pone} \left(\sqrt[\leftroot{2} \uproot{5} b]{\Pone; P_1}\right) \times_{\Pone} \left(\sqrt[\leftroot{2} \uproot{5} c]{\Pone; P_\infty}\right).
\end{equation*}
\end{definition}  

We start by summarizing some straightforward geometric properties of the \belyi
stack. See \cite[Definition 11.2.1]{Voight&Zureick-Brown22} and
\cite[Definition 5.2.1]{Voight&Zureick-Brown22} for the definition of a
(relative) stacky curve.
\begin{lemma}
    \label{lemma:properties-belyi-stack}
    The following statements hold.
    \begin{enumerate}[label=(\roman*),leftmargin=*]
    \item \label{item:belyi-stacky-curve} The \belyi stack has coarse space
      $\Pone_{\Z}$. The coarse space morphism $\Pone\abc \to \Pone$ is an
      isomorphism over the open set $U = \Pone-P_0\cup P_1 \cup P_\infty$.
    \item \label{item:tame-over-R} The \belyi stack is tame in the
      sense of \cite{Abramovich&Olsson&Vistoli08}. Moreover, if $R = \Z[1/abc]$,
      then the base change $\Pone\abc_R$ is Deligne--Mumford. In particular,
      $\Pone\abc_R$ is a relative stacky curve over $R$.
    \item \label{item:euler-char} For every geometric point
      $s\colon \Spec k \to \Spec R$, the fiber $\Pone\abc_s$ is a stacky curve
      over $k$. Moreover, the Euler characteristic of
      $\Pone\abc_s$ is
    \begin{equation*}
        \chi(\Pone\abc_s) = \tfrac1a + \tfrac1b + \tfrac1c - 1.
    \end{equation*}
    We define this common value to be the \cdef{Euler characteristic} of
    $\Pone\abc$.
    \end{enumerate}
\end{lemma}

We now turn to the arithmetic of the \belyi stack.

\begin{lemma}[$R$-points on the \belyi stack]
    \label{lemma:R-points-belyi-stack}
    Let $R$ be a principal ideal domain with fraction field $K$. Let
    $\Pone\abc$ be the base extension of the \belyi stack to $R$. The set
    $\Pone\abc\ideal{R}$ surjects onto the subset of
    $Q = (s:t) \in \Pone(R) = \Pone(K)$ such that $Q \in \brk{P_0,P_1,P_\infty}$, or:
    \begin{itemize}[leftmargin=*]
        \item $I(P_0,Q) = sR$ is an $a^\th$ power.
        \item $I(P_1,Q)= (s-t)R$ is a $b^\th$ power.
        \item $I(P_\infty,Q) = tR$ is a $c^\th$ power.
        \end{itemize}
        This map is injective away from $\brk{P_0,P_1,P_\infty}$; the fibers over
        $P_0$, $P_1$, and $P_\infty$ are torsors under $R^\times/(R^\times)^a$,
        $R^\times/(R^\times)^b$, and $R^\times/(R^\times)^c$ respectively.
\end{lemma}
\begin{proof}
  Let $\XX$ denote the Belyi stack. The set $\XX\ideal{R}$ is the fiber product
  of sets
  $$\left(\sqrt[a]{\Pone;P_0}\right)\ideal{R} \times_{\Pone(R)}
  \left(\sqrt[b]{\Pone;P_1}\right)\ideal{R} \times_{\Pone(R)}
  \left(\sqrt[c]{\Pone;P_\infty}\right)\ideal{R},$$ so the result follows from the
  description of the $R$-points of the $n^\th$ root stack of the projective
  line at a given point $P$ given in \Cref{prop:P1-rooted-at-P}.
\end{proof}

As Darmon observed in \cite[p.~5]{Darmon97}, the integral points on the \belyi
stack $\Pone\abc$ correspond to primitive integral solutions to generalized
Fermat equations of signature $\abc$, up to some \emph{sloppiness} in the signs
(i.e., $A,B,C \in \brk{\pm 1} = \Z^\times$). If we consider
$\Z[\mcS^{-1}]$-points instead, the same is true but allowing the coefficients
$A, B, C \in \Z[\mcS^{-1}]^\times$.
\begin{lemma}
 \label{lemma:belyi-points-are-solutions}
  Let $\mcS$ be a finite (possibly empty) set of primes, and let
  $R = \Z[\mcS^{-1}]$. Then, every point $Q \in \Pone\abc\ideal{R}$ arises as
  $j(x,y,z) = Q$ for some primitive integral solution $(x,y,z)$ to a
  generalized Fermat equation $F\colon \gfe = 0$, where $A, B, C \in R^\times$.
\end{lemma}

\begin{proof}
  Explicitly, an object in the groupoid $\XX(Q)$ is a triple
  \begin{equation*}
    (Q,\left[(M_0, \lambda_0),(M_1, \lambda_1),(M_\infty,
      \lambda_\infty)\right], (\sigma_0, \sigma_1,\sigma_\infty))
  \end{equation*}
  where $(M_0, \lambda_0),(M_1, \lambda_1),(M_\infty, \lambda_\infty)$ are
  generalized effective Cartier divisors on $\Spec R$, and $\sigma_0,
  \sigma_1,\sigma_\infty$ are isomorphisms of generalized effective Cartier
  divisors on $\Spec R$
  \begin{align*}
    \sigma_0\colon (M_0^{\otimes a},\lambda_0^{\otimes a}) &\to
                                                             Q^*(\OO_{\Pone}(-P_0), j_0) = (M(P_0,Q), \lambda(P_0, Q)),\\
    \sigma_1\colon (M_1^{\otimes b},\lambda_1^{\otimes b}) &\to
                                                             Q^*(\OO_{\Pone}(-P_1),j_1)
                                                             = (M(P_1,Q), \lambda(P_1, Q)),
    \\
    \sigma_\infty\colon (M_\infty^{\otimes c},\lambda_\infty^{\otimes c}) &\to
                                                                            Q^*(\OO_{\Pone}(-P_\infty),j_\infty) = (M(P_\infty,Q), \lambda(P_\infty, Q)).
  \end{align*}
  If $Q = (s:t)$, it follows from the assumption that $R$ is a principal ideal
  domain and \Cref{lemma:R-points-belyi-stack} that $s = -A\cdot x^a$, $t = C\cdot z^c,$ and $s-t = B\cdot y^b$ for
  some $A,B,C \in R^\times$ and $x,y,z\in R$. Since $-s + (s-t) + t = 0$, this
  implies that $Ax^a + By^b + Cz^c = 0$. Moreover, since $sR + tR = R$, we also
  have that $x^aR+z^cR = R$.
\end{proof}

\subsection{Triangle groups and \belyi maps}
\label{sec:triangle-groups-belyi-maps}

\begin{situation} \hfill
  \begin{itemize}[leftmargin=*]
  \item $k$ denotes a perfect field.
  \item $K$ denotes a number field, with ring of integers $\OO_K$.
  \item For any prime $\mfp \in \Spec \OO_K$, let $K_\mfp$ be the $\mfp$-adic
    completion of $K$, and let $\kk(\mfp)$ be the corresponding
    residue field.
  \item $Z_k$ denotes a \cdef{nice} (smooth, projective, geometrically
    integral) \cdef{curve} (separated scheme of finite type over a field),
    defined over $k$.
  \item $\phi\colon Z_k \to \Pone_k$ will denote a $k$-morphism.
  \end{itemize}
\end{situation}

The fundamental group of the thrice-punctured Riemann sphere
$\C\Pone - \brk{0,1,\infty}$ is free of rank two. It is generated by loops
$\gamma_0, \gamma_1, \gamma_\infty$ going around the punctures, subject to the
single relation $\gamma_0\gamma_1\gamma_\infty = 1$. Introducing the stackyness
imposes the further relations
$$\gamma_0^a = \gamma_1^b = \gamma_\infty^c = 1$$
on the generators. This is the fundamental group of the \belyi orbifold
$\Pone\abc(\C)$. The abstract group defined by these generators and relations
is the \cdef{triangle group} $\tribar\abc$. For more on this topic see
\cite[Section 2]{Clark&Voight19}, \cite[Chapter II]{Magnus74}. (More generally,
the fundamental groups of any orbifold curve can be calculated via van Kampen's
theorem \cite[Proposition 5.6]{Behrend&Noohi06}.)

\begin{definition}
	\label{def:belyi-map}
	Let $Z_k$ be a nice curve defined over a perfect field $k$. A
    \cdef{$k$-\belyi map} is a finite $k$-morphism
    $\phi\colon Z_k \to \Pone_k$ that is unramified outside
    $\brk{0, 1, \infty} \subset \Pone(k)$.
  \end{definition}

  \begin{remark}
    These remarkable covers of the projective line are named after the
    Ukrainian mathematician G. V. \belyi, who famously proved that a complex
    algebraic curve can be defined over a number field if and only if it admits
    a $\C$-\belyi map \cite{Belyi79, Belyi02}. For this reason, it is customary
    to require that $k \subset \C$ to use the term $\belyi$ map. We ignore this
    convention, and allow $k$ to have positive characteristic.
  \end{remark}

  Since $\pi_1(\Pone\abc(\C))$ is the triangle group $\tribar\abc$, the Riemann
  Existence Theorem guarantees that monodromy groups of Galois \belyi maps are
  always finite quotients of triangle groups.
\begin{definition}
    \label{def:galois-belyi-map}
    Let $\phi\colon Z_k \to \Pone_k$ be a $k$-\belyi map with automorphism
    $k$-group scheme $\Autsch(\phi)$. We say that $\phi$ is \cdef{geometrically
      Galois} with Galois group $G$ if the extension of function fields
    $\kk(Z_\kbar) \supset \kk(\Pone_\kbar)$ is Galois, with Galois group $G$.
    Equivalently, $\phi$ is geometrically Galois if the \cdef{monodromy group}
    $\Autsch(\phi)(\kbar)$ is isomorphic to $G$ and acts transitively on the
    fibers $\phi^{-1}(Q)$ for every $Q\in \Pone(\kbar)$. This is the case if and
    only if $\#\Autsch(\phi)(\kbar) = \# G = \deg \phi$.
  \end{definition}

  \begin{definition}
    The \cdef{signature} of a geometrically Galois $k$-\belyi map
    $\phi\colon Z_k \to \Pone_k$ is the triple $(e_0, e_1, e_\infty)$ where
    $e_P$ is the ramification index $e_\phi(z)$ of any critical point
    $z \in Z_k$ with critical value $P \in \brk{0,1,\infty}$. The \cdef{Euler
      characteristic} of $\phi$ is the quantity
    \begin{equation}
    \label{eq:euler-char}
    \chi(\phi) \colonequals \tfrac{1}{e_0} + \tfrac{1}{e_1} + \tfrac{1}{e_\infty} - 1.    
\end{equation} 
\end{definition}

As a consequence of the Riemann Existence Theorem, there exist Galois \belyi
maps of any signature. See \cite[Proposition 3.1]{Darmon&Granville95} and \cite[Lemma 2.5]{Poonen05}.

\begin{proposition}
    \label{prop:belyi-Galois-cover}
    For any positive integers $a,b,c > 1$, there exists a number field $K$ and
    a geometrically Galois $K$-\belyi map $\phi\colon Z_K \to \Pone_K$ of
    signature $(e_0,e_1,e_\infty) = \abc$. Let $g$ be the genus of $Z_K$, and $G$
    be the monodromy group of $\phi$. Then
    $2-2g = \deg\phi\cdot \chi(\phi)$. In particular,
    \begin{enumerate}[label=(\roman*)]
    \item If $\chi(\phi) > 0$, then $g = 0$ and
      $\deg \phi = \# G(\bar K) = 2/\chi(\phi)$.
    \item If $\chi(\phi) = 0$, then $g = 1$.
    \item If $\chi(\phi) < 0$, then $g > 1$.
    \end{enumerate}
  \end{proposition}

The definition of the \belyi stack implies the following.
\begin{lemma}
    \label{lemma:galois-etale-cover-belyi-stack}
    Let $\phi\colon Z_K \to \Pone_K$ be a geometrically Galois $K$-\belyi map
    of signature $\abc$. Then, there exists an {\'e}tale $\Autsch(\phi)$-torsor
    $\psi\colon Z_K \to \Pone\abc_K$ such that
    Diagram~(\ref{eq:belyi-etale-coarse}) commutes.
  \end{lemma}
    \begin{equation}
  \label{eq:belyi-etale-coarse}
  \begin{tikzcd}
      Z_K \arrow[dd, "\phi", "\text{geometrically Galois \belyi}"'] \arrow[rd, dashed, "\psi"', "\text{{\'e}tale
     } \Autsch(\phi)\text{-torsor}"] & \\
                   & \Pone\abc_K \arrow[dl, "\text{coarse}"] \\
      \Pone_K \, .
  \end{tikzcd}
\end{equation}

We wish to find integral models for our geometrically Galois \belyi maps
defined over number fields with certain good reduction properties. Informally,
we want to spread out \Cref{lemma:galois-etale-cover-belyi-stack} to the ring
of $\mathcal{T}$-integers in $\OO_K$ for a certain finite set of primes (containing
the archimedean primes). To accomplish this, we rely on the work of Beckmann
\cite{Beckmann89, Beckmann91}, which has been expanded and refined by \cite{Conrad00},
\cite{DebesGhazi11}, \cite{DebesGhazi12}, and more recently by
\cite{BalcikChanLiuViray25}.

\begin{lemma}[Good reduction]
  \label{lemma:good-reduction}
  Let $\phi\colon Z_K \to \Pone_K$ be a geometrically Galois $K$-\belyi map,
  with Galois group $G$ and signature $\abc$. Then, there exists a finite set of primes
  $\mathcal{T}$ in $K$ and a model $\Phi\colon Z \to \Pone_R$, defined
  over $R = \OO_{K}[\mathcal{T}^{-1}]$, such that for every $\mfp \not\in\mathcal{T}$:
  \begin{enumerate}[leftmargin=*]
  \item \label{it:good-red} $\phi$ has \cdef{good reduction at $\mfp$} (meaning that
    $\phi\times_K K_\mfp$ has good reduction in the sense of \cite[Definition
    4.1]{BalcikChanLiuViray25}), and
  \item \label{it:special-fiber} the special fiber $\Phi_\mfp\colon Z_{\kk(\mfp)} \to \Pone_{\kk(\mfp)}$
    is a geometrically Galois $\kk(\mfp)$-\belyi map with Galois group $G$ and
    signature $\abc$.
  \end{enumerate}
  Moreover, under these conditions, there exists an {\'e}tale
  $\Autsch(\Phi)$-torsor $\Psi\colon Z \to \Pone\abc_R$ such that
  Diagram~(\ref{eq:belyi-etale-coarse-R}) commutes.
  \end{lemma}
  
\begin{equation}
  \label{eq:belyi-etale-coarse-R}
  \begin{tikzcd}
    Z \arrow[dd, "\Phi"'] \arrow[rd, dashed, "\Psi"', "\text{{\'e}tale }\Autsch(\Phi)\text{-torsor}"] & \\
    & \Pone\abc_R \arrow[dl, "\text{coarse}"] \\
    \Pone_R \, .
  \end{tikzcd}
\end{equation}
\begin{proof}
  The existence of the finite set $\mathcal{T}$ for which (\ref{it:good-red}) holds is
  proved in \cite[Lemma 5.1]{BalcikChanLiuViray25}: note that if
  $L \supseteq K$ is the smallest field extension over which $\phi_L$ is
  Galois, then $Z_L \to Z_K \to \Pone$ is the Galois closure of $\phi$. For
  (\ref{it:special-fiber}) and the final statement, we can apply \cite[Theorem
  5.3]{BalcikChanLiuViray25}.
\end{proof}

\section{The stack \texorpdfstring{$[\mcU/\sfH]$}{[U/H]}}
\label{sec:fundamental-theorem}

\subsection{The group \texorpdfstring{$\sfH$}{H}}
\label{sec:H}
For this section we will need some basic notions from the theory of
diagonalizable group schemes of multiplicative type. See the notes of
Oesterl\'e \cite{Oesterle14} and Conrad \cite[Appendix B]{Conrad14}.

Given a base scheme $S$, and a finitely generated $\Z$-module $M$, we define
$\DD_S(M)$ to be the $S$-group scheme $\Spec \OO_S[M]$ representing the functor
$\underline{\Hom}_{\, S-\GrpSch}(M_S, \Gm)$ of characters of the constant
$S$-group scheme $M_S$. An $S$-group scheme is called \cdef{diagonalizable} if
it is isomorphic to $\DD_S(M)$ for some finitely generated $\Z$-module $M$.
Moreover, $\DD_S$ gives a contravariant functor between finitely generated
$\Z$-modules and the category of diagonalizable $S$-group schemes satisfying
certain exactness properties that are summarized in \cite[5.3]{Oesterle14}.

\begin{situation}
  \label{situation:H} Let
  \begin{itemize}[leftmargin=*]
  \item $\DD$ denote the functor described above, over the base scheme $S =
    \Spec \Z$.
  \item $\abc$ be a triple of positive integers.
  \item $m \colonequals \gcd(bc, ac, ab)$, and define the \cdef{weight
      vector} of $\abc$ by $\bfw = (w_0, w_1, w_\infty)$, where $w_0 = bc/m$,
    $w_1 = ac/m$ and $w_\infty = ab/m$.
  \item $\Gm(\bfw)$ be the image of the (injective) homomorphism $\Gm \to
    \Gm^3$ given by $\lambda \mapsto (\lambda^{w_0}, \lambda^{w_1},\lambda^{w_\infty})$.
  \item $\tribar\abc$ denote the triangle group
      \begin{equation*}
    \tribar\abc = \ideal{\gamma_0, \gamma_1, \gamma_\infty : \gamma_0^{a} =
      \gamma_1^{b} = \gamma_\infty^{c} = \gamma_0\gamma_1\gamma_\infty = 1}.
  \end{equation*}
  \end{itemize}
\end{situation}

\begin{definition}
  \label{def:H}
  Consider the finitely generated $\Z$-module
  \begin{equation}
    \label{eq:M}
    M \colonequals
  \ideal{(a,-b,0), (0,b,-c), (-a,0,c)} \subset \Z^3.
  \end{equation}
  Define $\sfH$ to be the
  subgroup $\DD(\Z^3/M)$ of $\Gm^3 = \DD(\Z^3)$.
\end{definition}

The diagonalizable group $\sfH$ admits a maximal torus corresponding to the
$\Z$-free part of $\Z^3/M$. Moreover, we have the following characterization.
An important formula to keep in mind is
\begin{equation}
  \label{eq:lcm-abc}
  \lcm(a,b,c) = \dfrac{abc}{\gcd(bc,ac,ab)}.
\end{equation}

\begin{lemma}[The structure of $\sfH$]
  \label{lem:H-structure} \hfill
  \begin{enumerate}[leftmargin=*]
  \item \label{it:Z3/M} The $\Z$-module $\Z^3/M$ is isomorphic to  $\Z \oplus
    \tribar\abc^\ab$.
  \item \label{it:K} Let $\sfK$ be the kernel of the map
    \[\mu_{a}\times \mu_{b}\times \mu_{c} \to \mu_{\lcm(a,b,c)}, \quad (\xi_0,\xi_1,\xi_\infty) \mapsto \xi_0\cdot \xi_1
      \cdot \xi_\infty .\]
    Then $\sfK \cong \DD(\tribar\abc^\ab)$.
  \item \label{it:H-structure} The group scheme
    $\sfH$ is equal to $\Gm(\bfw)\cdot \sfK$ and isomorphic to $\Gm\times \sfK$.
  \item \label{it:case-m=1} In particular, when $\gcd(bc,ac,ab) = 1$,
    $\sfH = \Gm(\bfw) \cong \Gm$.
  \end{enumerate}
\end{lemma}
\begin{proof}
  (\ref{it:Z3/M}) We calculate the invariant factor decomposition of $\Z^3/M$
  from the Smith normal form of the matrix having the generators of $M$ as its
  rows \cite[Theorem 2.3]{Stanley16}.
  Let \[\mathsf{m} = \begin{bmatrix}
      a & -b & 0 \\
      0 & b & -c \\
      -a & 0 & c
    \end{bmatrix}.\]
  From Stanley's formula \cite[Theorem
  2.4]{Stanley16}, we see that
  \[ \mathrm{SNF}(\mathsf{m})
     =
    \begin{bmatrix}
      d & 0 & 0 \\
      0 & m/d & 0 \\
      0 & 0 & 0
    \end{bmatrix},
  \]
  where $d = \gcd(a,b,c)$ is the greatest common divisor of the $1\times
  1$ minors, and $m = \gcd(bc,ac,ab)$ is the greatest common
  divisor of the $2\times 2$ minors. It follows that
  $\Z^3/M \cong \Z \oplus \Z/d\Z \oplus \Z/(m/d)\Z$.

  It remains to show that $\Z/d\Z \oplus \Z/(m/d)\Z$ is isomorphic to
  $\tribar\abc^\ab$. To this end, note that the group $\tribar\abc^\ab$ is
  isomorphic to the quotient of $\Z^3$ by the subgroup
  \[
    J = \ideal{(a,0,0), (0,b,0), (0,0,c), (1,1,1)}.
  \]
  As before, we calculate the invariant factor decomposition of $\Z^3/J$ via a
  Smith normal form computation.
   \[\mathrm{SNF}
    \begin{bmatrix}
      a & 0 & 0 \\
      0 & b & 0 \\
      0 & 0 & c \\
      1 & 1 & 1
    \end{bmatrix} =
        \begin{bmatrix}
      1 & 0 & 0 \\
      0 & d & 0 \\
      0 & 0 & m/d \\
      0 & 0 & 0
    \end{bmatrix} .
  \]
  We conclude that
  $\tribar\abc^\ab \cong \Z^3/J\cong \Z/d\Z \oplus \Z/(m/d)\Z $.

  (\ref{it:K}) From the presentation given in \Cref{situation:H}, we see
  that $\tribar\abc^\ab$ is the cokernel of the map
  $\Z/l\Z \to \Z/a\Z \oplus \Z/b\Z \oplus \Z/c\Z$ taking
  $1 \md l \mapsto (1 \md a, 1\md b, 1 \md c)$, where $l = \lcm(a,b,c)$. The
  result follows by applying the functor $\DD$.

  (\ref{it:H-structure}) The computation above shows that $\Z^3/M$ has
  $\Z$-rank one. The free part of $\Z^3/M$ corresponds to the (dual of the)
  kernel of the matrix $\mathsf{m}$. That is, we want a
  generator for the subgroup of $\mathbf{v} \in \Z^3$ such that
  \[
    \begin{bmatrix}
      a & -b & 0 \\
      0 & b & -c \\
      -a & 0 & c
    \end{bmatrix}\mathbf{v} =
    \begin{bmatrix}
      0\\0\\0
    \end{bmatrix}.
  \]

  In other words, we are looking for minimal $v_1, v_2, v_3 \in \Z$ satisfying
  $av_1 = bv_2 = cv_3$. But this is precisely the property defining the weight
  vector $\bfw$ (see \Cref{situation:H}). That
  $\sfH = \Gm(\bfw)\cdot \sfK \cong \Gm(\bfw)\times \sfK$ follow from \Cref{it:Z3/M} by
  applying the the functor $\DD$ to $\Z^3/M \cong \Z \oplus \tribar\abc^{\ab}$, and the
  general fact that $\DD(M_1\oplus M_2) \cong \DD(M_1)\times \DD(M_2)$ for arbitrary finitely
  generated $\Z$-modules $M_1, M_2$.
\end{proof}

\begin{lemma}
  \label{lemma:H1-H}
  Let $\mcS$ be a finite set of rational primes, and let $R = \Z[\mcS^{-1}]$.
  Then $\HH^1(R,\sfH_R)$ is finite. 
\end{lemma}
\begin{proof}
  From \Cref{lem:H-structure}\ref{it:H-structure}, we have
  $\sfH \cong \Gm \times \sfK$. Since $\HH^1$ takes finite products of fppf
  group schemes to products,
  \[ \HH^1(R,\sfH_R) \cong \HH^1(R,\Gm) \times \HH^1(R,\sfK_R)
     \cong \Pic R \times \HH^1(R,\sfK_R). \]
  The ring $R = \Z[\mcS^{-1}]$ is a principal ideal domain, so $\Pic R$ is
  trivial and $\HH^1(R,\sfH_R) \cong \HH^1(R,\sfK_R)$, which is finite.
\end{proof}

\begin{remark}
  Minkowski's theorem gives $\HH^1_{\et}(\Z,\sfK) = 0$, but $\sfK$ is not
  smooth over $\Z$ at primes dividing $\#\sfK(\Qbar)$, so the étale and fppf
  cohomology groups differ. For instance, taking $\abc = (2,2,2)$, we have
  $\sfK \cong \mu_2\times\mu_2$ and
  $\HH^1_{\fppf}(\Z,\mu_2) \cong \Z^\times/(\Z^\times)^2 \cong \Z/2\Z \neq 0$.
\end{remark}

% \begin{proof}
%   From $\sfH \cong \Gm \times \sfK$, we obtain the exact sequence
%   $\HH^1(R, \Gm) \to \HH^1(R, \sfH) \to \HH^1(R, \sfK) \to \HH^2(R, \Gm)$. Since
%   $\HH^1(R, \Gm) = \Pic R$ is trivial, we have that $\HH^1(R,\sfH)$ injects
%   into the finite group $\HH^1(R,\sfK)$.
% \end{proof}

\subsection{Proof of the main theorem}
\label{sec:proof-main-theorem}
We are ready to prove the main result.
\begin{situation} We place ourselves in the following situation for the remainder of this
section.
  \begin{itemize}[leftmargin=*]
  \item Let $F \colonequals \Spec \Z[\x,\y,\z]/\ideal{\gfe} \subset \A^3$.
  \item Let $\mcS$ be the set of primes dividing the integer
    $a\cdot b\cdot c\cdot A\cdot B\cdot C$.
  \item Let $R = \Z[\mcS^{-1}]$ be the ring of $\mcS$-integers.
  \item Let $\sfH $ be the affine group scheme introduced in
    \Cref{def:H}.
  \item Let $\mcU$ be the punctured cone associated to $F$, defined over $R$.
  \item Let $s\colon \Spec k \to \Spec R$ denote a geometric point.
  \item For a geometric object $\XX$ defined over $R$, we let $\XX_s
    \colonequals \XX \times_s \Spec k$ denote the geometric fiber above $s$.
  \end{itemize}
\end{situation}

We recall the statement of our main theorem and then proceed to prove some
preliminary lemmas.
\begin{theorem}
  \label{thm:re-main} The map
  \begin{equation}
    \label{eq:j-thm}
    j\colon \mcU \to \Pone_R, \quad (x,y,z) \mapsto (-Ax^a:Cz^c)
  \end{equation}
  induces an isomorphism $\bfj\colon [\mcU/\sfH_R] \cong \Pone\abc_{R}$.
\end{theorem}

The reason we are interested in the group scheme $\sfH$ is that it arises as
the stabilizer in $\Gm^3$ of the punctured cone $\mcU$ associated to a
generalized Fermat equation.
\begin{lemma}
  \label{lem:H=Stab}
  Let $\sfS$ be the stabilizer subgroup of $\mcU$ under the action of $\Gm^3$
  on $\A^3_\Z$. Then, $\sfH \subset \sfS$ and $\sfH_R = \sfS_R$.
\end{lemma}
\begin{proof}
  By definition, $\mathsf{S} \colonequals \Stab_{\Gm^3}(\mcU)$ is the group
  scheme that takes any $\Z$-algebra $B$ to the group
  \begin{align*}
    \mathsf{S}(B) &= \brk{(\lambda_0,\lambda_1,\lambda_\infty)\in (B^\times)^3 :
      F(\lambda_0 \x, \lambda_1 \y, \lambda_\infty\z)/F(\x,\y,\z) \in
                    B^\times}, \\
    \intertext{and this group visibly contains}
    \sfH(B) &= \brk{(\lambda_0,\lambda_1,\lambda_\infty)\in (B^\times)^3 : \lambda_0^{a} =
       \lambda_1^{b} = \lambda_\infty^{c}}.
  \end{align*}
  So we have an inclusion $\sfH \hookrightarrow \sfS$. For every geometric
  point $s\colon \Spec k \to \Spec R$, this inclusion pulls back to an equality
  $\sfS_s = \sfH_s$, so we conclude that $\sfS_R = \sfH_R$ by fpqc descent
  \spcite{02L4} and spreading out.
\end{proof}

We start by considering the situation on the geometric fibers.
\begin{lemma}
  \label{lemma:main-geometric-fibers}
  For every geometric point $s\colon \Spec k \to \Spec R$, the map
    \begin{equation}
    \label{eq:j-fiber}
    j\colon \mcU_s \to \Pone_s, \quad (x,y,z) \mapsto (-Ax^a:Cz^c)
  \end{equation}
  induces an isomorphism $\bfj_s \colon [\mcU_s/\sfH_s] \cong \Pone\abc_s$.
\end{lemma}
\begin{proof}
  We omit the subscript ``$s$\,'' and work over $k$ throughout. We start by
  showing that $j$ induces a coarse map $j\colon [\mcU/\sfH] \to \Pone$. Recall that
  $\mcR = k[\x,\y,\z]/\ideal{\gfe}$ is the coordinate ring of $F$. Consider the
  affine open $D(\z) \subset F$, with corresponding coordinate ring
  $\mcR[1/\z]$. Note that $\mcU\cap D(\z) = D(\z)$. Since
  $D(\z) = \Spec \mcR[1/\z]$ is affine, $\sfH$ is linearly reductive, and
  $[D(\z)/\sfH]$ is tame, the natural map
  $[\Spec \mcR[1/\z]/\sfH] \to \Spec \mcR[1/\z]^\sfH$ is a good moduli space and
  thus a coarse moduli space (see \cite[Theorem 13.2 and Remark
  7.3]{Alper13}). % Note that we didn't need to work over k, same argument to
                  % calculate the coarse space applies working over R.
  Now, we calculate that $\mcR[1/\z]^\sfH = k\left[\tfrac{-A\x^a}{C\z^c}\right]$.
  Applying the same argument to $D(\x)$, the result follows by glueing the maps
  \begin{align*}
    [\mcU\cap D(\x) / \sfH] &\to \Spec k\left[\tfrac{-C\z^c}{A\x^a}\right], \quad
    [\mcU\cap D(\z) / \sfH] &\to \Spec k\left[\tfrac{-A\x^a}{C\z^c}\right] 
  \end{align*}
  to obtain the coarse map $j\colon [\mcU/\sfH] \to \Pone$.

  We proceed to show that $[\mcU/\sfH] \cong \Pone\abc$. By definition of
  $\Pone\abc$ as an iterated root stack, the map $j\colon \mcU \to \Pone$ induces a
  map $\bfj \colon [\mcU/\sfH] \to \Pone\abc$. Indeed, the map $j\colon \mcU \to \Pone$ satisfies
  \[j^*\OO_{\Pone}(-P_0) = \mcL_0^a, \quad j^*\OO_{\Pone}(-P_1) = \mcL_1^b,
    j^*\OO_{\Pone}(-P_\infty) = \mcL_\infty^c,\] with
  $\mcL_0 = \x\cdot\OO_\mcU, \mcL_1 = \y\cdot\OO_\mcU$ and
  $\mcL_\infty = \z\cdot\OO_\mcU$, and this gives rise an object in
  $\Pone\abc(\mcU)$.

  Since $[\mcU/\sfH]\ideal{k} = \Pone(k) = \Pone\abc\ideal{k}$, and the map
  $[\mcU/\sfH](k) \to \Pone\abc(k)$ induces isomorphisms between the stabilizer
  groups of the stacky points
  \begin{align*}
    \Stab_\sfH(V(\x)) & \cong \mu_a(k), \\
    \Stab_\sfH(V(\y)) & \cong \mu_b(k), \\
    \Stab_\sfH(V(\z)) & \cong \mu_c(k).
  \end{align*}
  The result follows from \cite[Lemma 5.3.10(a)]{Voight&Zureick-Brown22}.
\end{proof}

\begin{proof}[Proof of \Cref{thm:re-main}]
  The $R$-morphism $j$ is surjective (this can be checked on geometric fibers
  by fpqc descent \spcite{02KV} and spreading out) and $\sfH_R$-invariant. From
  \Cref{lemma:induced-map-quotient-stacks}, this induces a morphism
  $[\mcU/\sfH_R] \to \Pone_R$, which factors through the coarse map
  $\Pone\abc_R \to \Pone_R$ by the definition of the \belyi stack. Both
  $\Pone\abc_R$ and $[\mcU/\sfH_R]$ are tame relative stacky curves. To
  calculate the coarse space $\mcU/\sfH$ of $[\mcU/\sfH]$, we use the same
  argument as in the proof of \Cref{lemma:main-geometric-fibers}.

  \begin{equation*}
    \begin{tikzcd}
      & \mcU \arrow[dd, near end, "j"] \arrow[ld] \arrow[rd] & \\
      {[\mcU/\sfH_R]} \arrow[rd, "\text{coarse}"'] \arrow[rr, gray, dashed,
      near end, "\bfj"] & & \Pone\abc_R \arrow[ld, "\text{coarse}"] \\
      & \Pone_R &
    \end{tikzcd}
  \end{equation*}

  In summary, we have a morphism $\bfj \colon [\mcU/\sfH]_R \to \Pone\abc_R$
  with the property that on each geometric fiber, the induced map on the coarse
  spaces $(\mcU/\sfH)_s \to \Pone_s$ is an isomorphism inducing a
  stabilizer-preserving bijection between $[\mcU/\sfH]\ideal{\kbar}$ and
  $\Pone\abc\ideal{\kbar}$. \cite[Lemma 5.3.10 (a)]{Voight&Zureick-Brown22}
  implies that $\bfj_s$ is an isomorphism for every geometric point
  of $\Spec R$, and this implies that the same is true globally. Alternatively,
  we can apply Santens' characterization of tame relative stacky curves
  \cite[Lemma 2.1]{Santens2023}.
\end{proof}

\section{The method of Fermat descent}
\label{sec:fermat-descent}

What follows is a brief discussion of the method of Fermat descent. This is not
intended to be a comprehensive algorithm for finding the primitive integral
solutions of an arbitrary generalized Fermat equation with integer
coefficients. Rather, it is meant as an artisanal guide to the method, from the
point of view developed in this article.

\begin{situation}
  We fix the following setup for the remainder of this section.
  \begin{itemize}[leftmargin=*]
  \item Let $\abc$ be a triple of positive integers, with $a,b,c > 1$.
  \item Let $F\colon \gfe = 0 \subset \A^3_\Z$ be a generalized Fermat
    equation.
  \item Let $\mcS$ be the set of primes $p$ dividing $a\cdot b\cdot c \cdot A
    \cdot B \cdot C$, and $R \colonequals \Z[\mcS^{-1}]$.
  \item Let $\mcU$ denote the punctured cone associated to $F$.
  \item Let $\sfH$ be as in \Cref{def:H}.
  \item Let $j\colon \mcU \to \Pone$ be the morphism $(x,y,z)\mapsto (-Ax^a:Cz^c)$.
  \end{itemize}
\end{situation}

As a consequence of \Cref{thm:re-main}, we have that $j(\mcU(\Z))$ is contained
in the set $\Pone\abc\ideal{R} \subset \Pone(R) = \Pone(\Q)$. This is our
starting point. The method of Fermat descent consists of three steps: covering,
twisting, and sieving.

\subsection{Covering}
\label{sec:covering}
The goal is to find a geometrically Galois \belyi map
(\Cref{def:galois-belyi-map}) $\phi\colon Z_K \to \Pone_K$. We know that one exists from
\Cref{prop:belyi-Galois-cover}, but we might have to base extend to a number
field $K$ to find it. Furthermore, the map $\phi$ admits an integral model
$\Phi$ over $\Spec \OO_{K}[\mathcal{T}^{-1}]$ for some finite set of primes
$\mathcal{T}$ which we can arrange to contain the primes above $\mcS$. In
practice, it is desirable to minimize the complexity of this data as much as
possible.

\subsection{Twisting}
\label{sec:twisting}
Now that we found a covering $\Phi\colon Z \to \Pone_{\OO_K[\mathcal{T}^{-1}]}$
with automorphism group scheme $\Autsch(\Phi)$, we are tasked with finding:
\begin{enumerate}
\item \label{it:H1} The (\v{C}ech fppf) cohomology set $\HH^1(\OO_K[\mathcal{T}^{-1}],
  \Autsch(\Phi))$. 
\item \label{it:twists} For each $\tau \in \HH^1(\OO_K[\mathcal{T}^{-1}], \Autsch(\Phi))$, the twist
  $\Phi_\tau\colon Z_\tau \to \Pone$.
\end{enumerate}

For (\ref{it:H1}), the first observation is that there is an isomorphism of pointed
sets between $\HH^1(\OO_K[\mathcal{T}^{-1}], \Autsch(\Phi))$ and the Galois cohomology set
$\HH^1_\mathcal{T}(K, \Gal(\phi))$, parametrizing isomorphism classes of Galois
{\'e}tale $K$-algebras with Galois group $\Gal(\phi)$
$= \Autsch(\phi)(\bar K) = \Aut(\phi_{\bar K}) $, and unramified outside
$\mathcal{T}$. The second observation is that we might not need the full
cohomology set. More precisely, for our purposes of finding $\mcU(\Z)$, we want
to identify the subset of those $\tau$ for which $\Phi_\tau(Z_\tau(K)) \cap j(\mcU(\Z)) \neq \emptyset$.

\subsection{Sieving}
\label{sec:sieving}
Let $R' = \OO_K[\mathcal{T}^{-1}]$. If we reach this step, we have found a subset $T \subset
\HH^{1}(R', \Autsch(\Phi))$ such that
\begin{equation*}
  j(\mcU(\Z)) \subset \bigcup_{\tau \in T}\Phi_\tau(Z_\tau(R')) = \bigcup_{\tau \in T}\phi_\tau(Z_\tau(K)).
\end{equation*}
We use the word ``sieve'' to mean that our goal is to separate the points on
the right hand side that do not come from primitive integral solutions to $F$.
For this purpose, it might be helpful to recall that $j(\mcU(\Z))$ is also
contained in $\Pone\abc\ideal{R}$, where $R = \Z[\mcS^{-1}]$ is a principal ideal
domain, to apply \Cref{lemma:R-points-belyi-stack}.

\providecommand{\bysame}{\leavevmode\hbox to3em{\hrulefill}\thinspace}
\providecommand{\MR}{\relax\ifhmode\unskip\space\fi MR }
\providecommand{\MRhref}[2]{%
  \href{http://www.ams.org/mathscinet-getitem?mr=#1}{#2}
}
\providecommand{\href}[2]{#2}

\end{document}